\documentclass{amsart}[12pt]
\usepackage{amsmath, amsthm, amscd, amsfonts}
\usepackage[english]{babel}
\usepackage{amsfonts}
\usepackage{amsthm}
\usepackage{graphicx,float}
\usepackage{amssymb}
\usepackage{ mathrsfs }
\usepackage{comment}
\usepackage{amssymb}
\usepackage{cancel}
\usepackage{comment}
\usepackage{color}
\usepackage{ dsfont }

\newcommand{\Cn}{$C^{(n)}$\,}
\newcommand{\Ult}{\mathrm{Ult}}

\newcommand{\Con}{\mathrm{Con}}
\newcommand{\VP}{$\mathrm{VP}\,$}
\newcommand{\Co}{$C^{(<\omega)}$\,}

\newcommand{\len}{\text{len}\,}
\newcommand{\ran}{\text{range}\,}
\newcommand{\supp}{\text{supp}\,}
\newcommand{\pikappalambda}{\mathcal{P}_\kappa(\lambda)}
\newcommand{\ZFC}{\mathrm{ZFC}\,}
\newcommand{\GCH}{\mathrm{GCH}\,}
\newcommand{\crit}{\mathrm{crit}}
\newcommand{\cf}{\mathrm{cf}}
\newcommand{\otp}{\mathrm{otp}}
\newcommand{\dom}{\text{dom}\,}
\newcommand{\Add}{\mathrm{Add}}

\newtheorem{theo}{Theorem}[section]
\newtheorem{quest}{Question}

\newtheorem{defi}[theo]{Definition}
\newtheorem{axiom}{Axiom}
\newtheorem{lemma}[theo]{Lemma}
\newtheorem{prop}[theo]{Proposition}
\newtheorem{remark}[theo]{Remark}
\newtheorem{claim}[theo]{Claim}
\newtheorem{cor}[theo]{Corollary}

\begin{document}
\title[Identity Crisis between supercompact and V\v{o}penka's]{Identity Crisis between supercompactness and V\v{o}penka's Principle}
\author{Yair Hayut}
\address{School of Mathematical Sciences. Tel Aviv University. Tel Aviv 69978, Israel}
\email{yair.hayut@mail.huji.ac.il}
\author{Menachem Magidor}
\address{Institute of Mathematics, The Hebrew University of Jerusalem, Jerusalem 91904, Israel}
\email{mensara@savion.huji.ac.il}
\author{Alejandro Poveda}
\address{Departament de Matemàtiques i Informàtica, Universitat de Barcelona. Gran Via de les Corst Catalanes, 585, Barcelona 08007, Catalonia.}
\email{alejandro.poveda@ub.edu}
\thanks{The third author's research has been supported by MECD (Spanish Government) Grant no FPU15/00026, MEC project number MTM2017-86777-P and SGR (Catalan Government) project number 2017SGR-270.}


\begin{abstract}
In this paper we study the notion of $C^{(n)}$-supercompactness introduced by Bagaria in \cite{Bag} and prove the identity crises phenomenon for such class. Specifically, we show that consistently the least supercompact is strictly below the least $C^{(1)}$-supercompact  but also that the least supercompact is $C^{(1)}$-supercompact (and even $C^{(n)}$-supercompact). Furthermore, we prove that under suitable hypothesis that the ultimate identity crises is also possible. These results solve several questions posed by Bagaria and Tsaprounis.
\end{abstract}

\maketitle

\section{Introduction}
Reflection principles are one of the most important and ubiquitous phenomena in mathematics. Broadly speaking one can  formulate reflection principles by means of the slogan \textit{``If a structure enjoys some property, there is a smaller substructure satisfying the same property''}. In practice the term \textit{smaller substructure} use to be modulated by some given regular cardinal.

The dual version of reflection principles are the so called the compactness principles. The way of defining any compactness principle is by means of the slogan \textit{``If every small substructure of a given structure enjoys some property, then the structure also satisfies the property''}. One can easily translate any reflection principle to a compactness one and conversely, hence the choice for the formulation of a given problem will depend exclusively  on which of them is more illustrative. Mathematical Logic, and specially Set Theory, is one of those fields where most of the central questions admit a suitable formulation in terms of reflection principles and thus its study becomes of special interest. Among many other examples, we can highlight the investigations on stationary reflection or the study of the tree property at regular cardinals.

From a platonistic perspective, Set Theory is essentially the field devoted to reveal the truths of the universe of sets. Long time ago L\'evy and Montague proved the Reflection theorem (see e.g. \cite{Kun}) discovering that reflection is an essential feature of the model-theoretic architecture of $V$. More precisely, for each metatheoretic $n\in\omega$, they proved that the class of ordinals $\alpha\in C^{(n)}$ such that $V_\alpha\prec_n V$ is a proper club class. Little time after, L\'evy noticed that the Reflection theorem is equivalent to the axioms of Infinity and Replacement modulo the remaining $\mathrm{ZF}$ axioms; accentuating, even more, the belief that reflection is one of the cornerstones of Set Theory.

One of the ways reflection principles have became more and more sophisticated by means of the machinery of elementary embedding. Many of the well-known large cardinals notions are formulated as critical points of elementary embeddings $j:V\rightarrow M$ between the universe and some transitive substructure $M\subseteq V$. Morally the family of large cardinals correspond to a hierarchy of principles asserting that there are strong forms of agreement between the whole universe $V$ and certain substructures of it. The degree of agreement between the two reals depends on the specific properties of $j$.

The purpose of the present paper is to contribute to the investigation of the identity crises phenomenon in the section of the large cardinal hierarchy ranging between the first supercompact cardinal and Vopenka's Principle (\VP$\,$on the sequel). These cardinals are known as \Cn-cardinals and were introduced by Bagaria in \cite{Bag} aiming  for a sharp study of the strongest forms of reflection. Morally these families of large cardinal principles stablish the canonical way to climb upwards in the ladder towards the \textit{ultimate reflection principle}. For convenience throughout the paper we shall denote by $\mathfrak{M},\mathfrak{K},\mathfrak{S}$, $\mathfrak{S}_{\omega_1}$\footnote{A cardinal $\kappa$ is called $\omega_1$-strongly compact if for every set $X$ and every $\kappa$-complete filter over $X$, there is some $\omega_1$-complete ultrafilter extending it. For a extensive study of such cardinals see \cite{BM} and \cite{BM2}. } and $\mathfrak{E}$ the classes of measurable, strongly compact, $\omega_1$-strongly compact, supercompact and extendible cardinals, respectively and by $\mathfrak{S}^{(n)}$ and $\mathfrak{E}^{(n)}$ the families of \Cn-supercompact and \Cn-extendible cardinals, respectively. Any non defined notion may be consulted in the excellent PhD dissertation of Tsaprounis \cite{TsaPhD}.

Several studies on the topic of \Cn-cardinals have been carried out succesfully by Bagaria and Tsaprounis whom investigations covers a broad spectrum embracing from the interplay of \Cn-cardinals with forcing to applications to Category theory and {Resurrection \-Axioms} (see \cite{Bag}\cite{BagEtAl} \cite{TsaChain}\cite{Tsan}\cite{Tsa}\cite{TsaResu}. Nonetheless,  there is a natural notion within the setting of the \Cn-cardinals which remains elusive and mysterious: \Cn-supercompactness.
\begin{defi}[\Cn-supercompactness \cite{Bag}]
A cardinal $\kappa$ is $\lambda$-\Cn-supercompact for some $\lambda>\kappa$, if there is an elementary embedding $j: V\rightarrow M$ such that $\crit(j)=\kappa$, $j(\kappa)>\lambda$, $M^\lambda\subseteq M$ and $j(\kappa)\in C^{(n)}$. A cardinal $\kappa$ is \Cn-supercompact if it is $\lambda$-\Cn-supercompact, for each $\lambda>\kappa$.
\end{defi}
Our purpose along the paper will be basically to answer the next three questions posed by Bagaria and Tsaprounis. 
\begin{quest}\label{Question1}
Are the notions of supercompactness and $C^{(1)}$-supercompactness equivalent? More generally, given $n\geq 1$, is it true that the first supercompact is the same as the first \Cn-supercompact?
\end{quest}
\begin{quest}\label{Question3}
Do the classes of \Cn-supercompact cardinals form a strictly increasing \newline hierarchy?
\end{quest}
\begin{quest}\label{Question2} 
Let $n\geq 1$. Is it the first \Cn-supercompact cardinal the first \Cn-extendible?
\end{quest}

Our contribution to settle the aforementioned questions can be summarized by the following two results:
\begin{theo}[Main Theorem 1]\label{supercompnotc1}
Assume $\mathrm{GCH}$ holds and let $\kappa$ be a supercompact cardinal. Then there is a generic extension $V^\mathbb{P}$ where $\kappa$ is still supercompact but not $C^{(1)}$-supercompact. In fact there is no elementary embedding  in $V^\mathbb{P}$, $j:V^\mathbb{P}\rightarrow M$, such that $\crit(j)=\kappa$, $M^\omega\subseteq M$ and $j(\kappa)$ being a limit cardinal.
\end{theo}

\begin{theo}[Main Theorem 2]\label{MainTheorem}
Let $n\geq 1$, $\kappa$ be a \Cn-supercompact cardinal and $\ell:\kappa\rightarrow \kappa$ be a $\mathfrak{S}^{(n)}$-fast function on $\kappa$. Then in the generic extension $V^\mathbb{M}$ given by a Magidor product of Prikry forcings $\kappa$ remains \Cn-supercompact and in fact it is the first ($\omega_1$-)strongly compact. In particular, the following holds in $V^\mathbb{M}$:
$$\min\mathfrak{M}<\min\mathfrak{K}_{\omega_1}=\min \mathfrak{K}=\min \mathfrak{S}=\min \mathfrak{S}^{(n)}<\min\mathfrak{E}.$$
\end{theo}

Both theorem \ref{supercompnotc1} and theorem \ref{MainTheorem} settle in a negative way the former questions. Furthermore building on the ideas developed for their respective proofs we shall show how to prove the following strengthenings:
\begin{theo}\label{StrengthTheorem1}
Assume $\GCH$ holds and that there are two supercompact cardinals with a $C^{(1)}$-supercompact cardinal above them. Then there is a generic extension of the universe where the following holds:
$$\min\mathfrak{M}<\min\mathfrak{K}<\min\mathfrak{S}<\min\mathfrak{S}^{(1)}.$$
\end{theo}  
\begin{theo}[The ultimate identity crises]\label{MainTheorem2}
Let $\langle V,\in,\kappa\rangle$ be a model of (large enough fragment of) $\mathrm{ZFC}^\star$ plus $C^{(<\omega)}-\mathrm{EXT}$. Then in the generic extension $V^\mathbb{M}$ it is true that
$$\min\mathfrak{M}<\min\mathfrak{K}_{\omega_1}=\min\mathfrak{K}=\min\mathfrak{S}=\min \mathfrak{S}^{(<\omega)}<\min\mathfrak{E}. $$
\end{theo}
The notions $C^{(<\omega)}-\mathrm{EXT}$ and $\mathfrak{S}^{(<\omega)}$ will be introduced at the end of section 3.


The structure of the paper is as follows.
Section \ref{Section 2} will be devoted to the proofs of theorems \ref{supercompnotc1} and \ref{StrengthTheorem1} while section \ref{Section 3} will be focused on the proofs of theorems \ref{MainTheorem} and \ref{StrengthTheorem1}. We shall end the paper with section 4 and section 5 where we respectively describe what is known up to the moment about \Cn-supercompact cardinals and what are the possibles futures for the research of this topic. All the notions and notations are quite standard and can be easily found either in general manuals or in the bibliography quoted below.
\section{The first $C^{(1)}$-supercompact can be greater than the first supercompact. }\label{Section 2}
The present section is devoted to the proof of theorems \ref{supercompnotc1} and \ref{StrengthTheorem1}. In particular, both results answer negatively Question \ref{Question1}. Before beginning with the details let us give a taste of the ideas 
involved in the proof of these results.

A classical theorem of Solovay  asserts that if a cardinal $\kappa$ is strongly compact (hence supercompact) then $\square_\lambda$ fails, all $\lambda\geq\kappa$ \cite{Sol}. More generally if $\kappa$ is a supercompact cardinal then $\square_{\lambda,\cf(\lambda)}$ fails, for $\cf(\lambda)<\kappa<\lambda$ (see proposition \ref{NotC1Further}).
Therefore it is then natural to ask \textit{how much square} can hold below a supercompact cardinal. Working in this direction Apter proved in \cite{Apter} the consistency of a supercompact cardinal with the existence of $\square_\lambda$-sequences for each cardinal $\lambda$ in a certain stationary subset of $\kappa$. On this respect it is worth to emphasize that this result is close to be optimal since there is no club $C\subseteq\kappa$ where $\square_\lambda$ holds, for each $\lambda\in C$. Indeed, let us assume aiming for a contradiction that $\kappa$ is supercompact and $C\subseteq \kappa$ is a club whit the above property. Let $U$ be the standard normal measure derived by some elementary embedding with critical point $\kappa$ and $M$ be the correspoding ultrapower. By normality of the measure $C\in U$, hence $\square_\kappa$ holds in $M$, and furthermore it is not hard to show that $(\kappa^+)^M=\kappa^+$. Altogether one has that $\square_\kappa$ holds, yielding to a contradiction with the supercompactness of $\kappa$.

Broadly speaking, the main point to kill the $C^{(1)}$-supercompactness of a supercompact cardinal $\kappa$ is to construct a generic extension where any elementary embedding witnessing the $C^{(1)}$-supercompactness of $\kappa$ would yield to the existence of a $\square_\lambda$-sequence above $\kappa$. To implement this idea one needs to force \textit{many} square sequences below $\kappa$ and afterwards argue that this is upwards reflected by any $C^{(1)}$-supercompact embedding with critical point $\kappa$. This is interesting since it points out that despite the existence of many squares sequences is not an inconvenience for supercompactness it does for $C^{(1)}$-supercompactness.

Our forcing construction will be an Easton support iteration guided by some Laver function on $\kappa$ of the canonical forcings for adding square sequences. Once one proves that this forcing is harmless with respect to the supercompactness of $\kappa$ it is not hard to prove that there are no witnesses for $C^{(1)}$-supercompactness in the generic extension. In particular theorem \ref{supercompnotc1}  yields to the next result of consistency:

\begin{cor}\label{CorConsist1}
$\Con(\ZFC+\GCH+\exists \kappa,\lambda\,(\kappa,\lambda\in\mathfrak{S}^{(1)}))$ implies $\Con(\ZFC+\GCH+\min\mathfrak{S}<\min\mathfrak{S}^{(1)})$.
\end{cor}
Working on the ideas needed for the proof of theorem \ref{supercompnotc1} we will show in subsection 2.2 how to use them to prove theorem \ref{StrengthTheorem1}.  As before, this result will automatically yield to the following consistency result:
\begin{cor}\label{CorConsist2}
$\Con(\ZFC+\GCH+\exists\kappa,\lambda\in\mathfrak{S}\,\exists\mu\in\mathfrak{S}^{(1)}(\lambda<\kappa<\mu))$ implies
$\Con(\ZFC + \min\mathfrak{M}<\min\mathfrak{K}<\min\mathfrak{S}<\min \mathfrak{S}^{(1)})$.
\end{cor}  

\subsection{The proof of theorem \ref{supercompnotc1}}
Let us start recalling some basic notions that are nece\-ssary for the proof of theorem \ref{supercompnotc1}.
\begin{defi}[$\square$-sequences]
Let $\mu\leq \kappa$ be two cardinals. A $\square_{\kappa,\mu}$-sequence is a sequence of sets $\vec{\mathcal{C}}=\langle\mathcal{C}_\alpha:\,\alpha\in \mathrm{Lim}\cap \kappa^+\rangle$\footnote{Here $\mathrm{Lim}$ denotes the class of all limit ordinals.} such that the following properties hold:
\begin{enumerate}
\item[(a)] For each $\alpha\in\mathrm{Lim}\cap \kappa^+$ the set $\mathcal{C}_\alpha$ is a family of club sets on $\alpha$ with $1\leq |\mathcal{C}_\alpha|\leq\mu$.
\item[(b)] For each $\alpha\in\mathrm{Lim}\cap \kappa^+$ with $\cf(\alpha)<\kappa$ the family $\mathcal{C}_\alpha$ only contains sets $C$ with $\mathrm{otp}(C)<\kappa$.
\item[(c)] For each $\alpha\in\mathrm{Lim}\cap \kappa^+$, the family $\langle\mathcal{C}_\beta:\,\beta\in \mathrm{Lim}\cap \alpha\rangle$ is coherently disposed; namely,
$$\forall C\in\mathcal{C}_\alpha\,\forall\beta\in \mathrm{Lim}(C)\, C\cap \beta\in\mathcal{C}_\beta.$$
\end{enumerate}
We shall say that $\square_{\kappa,\mu}$ holds if there is a $\square_{\kappa,\mu}$-sequence. Similarly, we will say that $\square_{\kappa, <\mu}$ holds if $\square_{\kappa,\theta}$ holds, for each $\theta<\mu$. We shall denote by $\square_\kappa$ and by $\square_\kappa^*$ the principles $\square_{\kappa, 1}$ and $\square_{\kappa,\kappa}$, respectively.
\end{defi}
There is a canonical forcing for adding a $\square_{\lambda,\mu}$-sequence by approximations but for the purposes of the current paper it will be enough to present the definition of the forcing for adding a $\square_\lambda$-sequence.
\begin{defi}\label{ForcingSquare}
Let $\lambda$ be an uncountable cardinal. The canonical poset for forcing a $\square_\lambda$-sequence $\mathbb{P}_{\square_\lambda}$ is the set of conditions $p$ such that
\begin{itemize}
\item[(a)] $p$ is a function with $dom(p)=(\alpha+1)\cap Lim$ with $\alpha\in\lambda^+\cap Lim$.
\item[(b)] For every $\beta\in dom(p)$, $p(\beta)\subseteq\beta$ is a club subset with $otp(p(\beta))\leq\lambda$.
\item[(c)] If $\beta\in dom(p)$ $\forall\gamma\in p(\beta)\cap Lim$ $(p(\gamma)=p(\beta)\cap \gamma)$.
\end{itemize}
endowed with the reverse end-extension order.
\end{defi}
Standard arguments show that $\mathbb{P}_{\square_\lambda}$ is a $(\lambda+1)$-strategically closed forcing (see \cite{CumIter}) and under $\mathrm{GCH}$, since $|\mathbb{P}_{\square_\lambda}|=\lambda^+$, it preserves cofinalities and respects the $\mathrm{GCH}$ pattern.

Many times it is helpful for carrying out lifting arguments that our iteration is defined in a  \textit{sparse enough set} of cardinals. The standard setting for such kind of arguments is described by a forcing iteration $\mathbb{P}$, an elementary embedding $j: V\rightarrow M$ and a factorization of the form $j(\mathbb{P})\cong\mathbb{P}\ast \dot{\mathbb{Q}}$. Under these conditions one expects that $\dot{\mathbb{Q}}$ enjoys of some closure property that helps to find  a $\dot{\mathbb{Q}}$-generic filter over $M^\mathbb{P}$. For instance, if $\mathbb{Q}$ is closed enough in $M^\mathbb{P}$ it is usual to build such a generic filter by means a diagonalization argument.

One of the standard procedures to build such iterations consist in guiding the iteration with a function $\ell$ presenting some \textit{fast behaviour}. Despite that we will need to consider slightly more general fast functions (see the preliminary discussion of Section 3), in this part we will only be interested in the case where $\ell$ is a Laver function. Recall that if $\kappa$ is a supercompact cardinal a function $\ell: \kappa\rightarrow V_\kappa$ is called a Laver function if for every $\lambda>\kappa$ there is a $\lambda$-supercompact elementary embedding $j:V\rightarrow M$  with $j(\ell)(\kappa)>\lambda$ \cite{Lav}.


Without loss of generality we may and do assume that the domain of $\ell$ is the club set of closure points $\alpha$ of $\ell$ (i.e. $\ell''\alpha\subseteq V_\alpha$) that are also strong limit cardinals.
\begin{defi}
Let $\mathbb{P}^\ell_\kappa$ be the $\kappa$-Easton support iteration\footnote{Namely, direct limits are taken at inaccessible cardinals and inverse limits elsewhere.} where $\mathbb{P}^\ell_0$ is the trivial forcing and for each ordinal $\alpha<\kappa$, if $\alpha\in dom(\ell)\cap E^\kappa_\omega$ and  $\Vdash_{\mathbb{P}^\ell_\alpha}\text{``}\check{\alpha}^+\text{ is a cardinal''}$ then $\Vdash_{\mathbb{P}^\ell_\alpha}\text{``}\mathbb{\dot{Q}}_\alpha=\mathbb{P}_{\square_\alpha}\text{''}$ and $\Vdash_{\mathbb{P}^\ell_\alpha}\text{``}\mathbb{\dot{Q}}_\alpha\text{ is trivial''}$, otherwise. 
\end{defi}
The next proposition shows that $\mathbb{P}^\ell_\kappa$ forces a $\square_\lambda$-sequence for each $\lambda\in \dom(\ell)\cap E^\kappa_\omega$ and thus $\square_\lambda$ holds in a stationary subset of $\kappa$.
\begin{prop}
Assume $GCH$. The iteration $\mathbb{P}^\ell_\kappa$ preserves cardinals, the $\mathrm{GCH}$ pattern and yields to a generic extension $V^{\mathbb{P}^\ell_\kappa}$ where $\square_\lambda$ holds, for all cardinal $\lambda\in E^\kappa_\omega\cap dom(\ell)$. 
\end{prop}

\begin{proof}
The first claim easily follows from the comments after definition \ref{ForcingSquare} so it is enough to prove the claim about the $\square_\lambda$-sequences. Let $\lambda\in dom(\ell)\cap E^\kappa_\omega$ be a cardinal and notice that $\mathbb{P}^\ell_\kappa$  factorizes as $ \mathbb{P}^\ell_{\lambda+1}\ast \dot{\mathbb{P}}^\ell_{tail}$, where $\dot{\mathbb{P}}^\ell_{tail}$ is some $\mathbb{P}^\ell_{\lambda+1}$-name for a $\lambda^+$-strategically closed iteration. Now notice that $\mathbb{P}^\ell_\lambda$ is $\lambda^+$-cc, hence $\mathbb{P}^\ell_{\lambda+1}$ forces $\square_\lambda$, and  $\mathbb{P}^\ell_{tail}$ preserves $(\lambda^+)^{V^{\mathbb{P}^\ell_{\lambda+1}}}$ so $\Vdash_{\mathbb{P}^\ell_\kappa}\text{``$\square_\lambda$ holds''}$.
\end{proof}
\begin{prop}\label{PreservationOfSuper}
Forcing with $\mathbb{P}^\ell_\kappa$ preserves the supercompactness of $\kappa$. Moreover in $V^{\mathbb{P}^\ell_\kappa}$, $\kappa$ is the first supercompact.
\end{prop}
\begin{proof}
The last claim follows immediately from the result of Solovay, \cite{Sol}.
Working in $V$, let $\lambda>\kappa$, $\theta=(2^{\lambda^{<\kappa}})^+$ and  $j:V\rightarrow M$ be some $\theta$-supercompact embedding such that $j(\ell)(\kappa)>\theta$ and $G\subseteq \mathbb{P}^\ell_\kappa$ a generic filter over $V$. First of all, since $j(\ell)\upharpoonright\kappa=\ell$, the forcing $j(\mathbb{P}^\ell_\kappa)$ factorizes as 
\[j(\mathbb{P}^\ell_\kappa)\cong \mathbb{P}^\ell_\kappa\ast \mathbb{Q}\ast \mathbb{P}_{tail}.\]

where $\mathbb{Q}$ is forced to be the trivial poset because $\cf^M(\kappa)>\omega$. On the other hand, $\Vdash_{\mathbb{P}^\ell_\kappa}\text{``$ \mathbb{Q}\ast\mathbb{P}_{tail}$ is $\theta$-strategically closed''}$ since $j(\ell)(\kappa)>\theta$. For the ease of notation we shall denote by $\mathbb{P}^*_{tail}$ the iteration $\mathbb{Q}\ast \mathbb{P}_{tail}$.
 The conditions in $\mathbb{P}^\ell_\kappa$ have bounded support in $\kappa$, hence $j\upharpoonright \mathbb{P}^\ell_\kappa=id$,  so $j''G\subseteq G\ast H$, for any $H\subseteq ({\mathbb{P}^*_{tail}})_G$ generic filter over $M[G]$. Set
$j^*:V[G]\rightarrow M[G\ast H]\subseteq V[G\ast H]$ be the corresponding lifting.
Since $\kappa$ is a Mahlo cardinal and $\mathbb{P}^\ell_\kappa$ is a $\kappa$-Easton support iteration of forcings in $V_\kappa$ the iteration $\mathbb{P}^\ell_\kappa$ is $\kappa$-cc and thus $M[G]$ remains closed by $\theta$-sequences. Similarly, since $({\mathbb{P}^*_{tail}})_G$ is $\theta^+$-strategically closed in $M[G]$ and $M[G]^\theta\subseteq M[G]$, one may argue that $M[G\ast H]$ is closed under $\theta$-sequences and  that $({\mathbb{P}^*_{tail}})_G$ is also $\theta$-strategically closed in the $V[G]$.

Working in the generic extension $V[G\ast H]$, it is straightforward to show that
$$X\in\mathcal{U}\,\longleftrightarrow\, X\subseteq (\pikappalambda)^{V[G]}\,\wedge\, j''\lambda\in j(X)$$
defines a $\lambda$-supercompact measure over $\pikappalambda^{V[G]}$. By standard arguments of counting nice names it can be checked that $\mathcal{U}$ has cardinality less than $\theta$. On the other hand, $({\mathbb{P}^*_{tail}})_G$ is $\theta^+$-strategically closed in $V[G]$ and thus the measure $\mathcal{U}$ was not introduced by the forcing $({\mathbb{P}^*_{tail}})_G$. Altogether this argument shows that $\mathcal{U}\in V[G]$; hence $\kappa$ is $\lambda$-supercompact in $V[G]$. Provided that $\lambda$ was chosen arbitrarily we have already proved that $\kappa$ remains fully supercompact after forcing with $\mathbb{P}^\ell
_\kappa$.
\end{proof}
We are now in conditions to prove theorem \ref{supercompnotc1}:
\begin{proof}[Proof of theorem \ref{supercompnotc1}]
For the rest of the proof fix $G\subseteq \mathbb{P}^\ell_\kappa$ a generic filter over $V$. Aiming for a contradiction let us assume that there is a supercompact embedding $j:V[G]\rightarrow M$ with $\crit(j)=\kappa$, $M^\omega\subseteq M$ and $j(\kappa)$ being a limit cardinal. Appealing to the closure properties of $M$ and to the elementarity of $j$ it is not hard to realize that the cardinal $j(\kappa)$ has uncountable cofinality in $V[G]$ and that there is a $\square_\lambda$-sequence in $M$, for each $\lambda\in j(\dom(\ell)\cap E^\kappa_\omega)$.

 On the other hand, since $\cf(j(\kappa))>\omega$, $E^{j(\kappa)}_\omega$ is a stationary set in $V[G]$ and thus also the set $j(E^\kappa_\omega \cap \dom(\ell))$. Let $\lambda\in j(E^\kappa_\omega \cap \dom(\ell))$ be some ordinal greater than $\kappa$ and notice that, of course, $\square_\lambda$ holds in $M$. Nevertheless we shall prove that this is also the case in $V[G]$ to yield to the desired contradiction. Aiming for this, it will be sufficient with proving that $M$ and $V[G]$ agree on the computations of the successor of $\lambda$: namely, $\left(\lambda^+\right)^{V[G]}=\left(\lambda^+\right)^M$. 
Since GCH holds in $V[G]$, hence also in $M$, and $M$ is closed by $\omega$-sequences, $({}^\omega\lambda)^M= {}^\omega \lambda$, $\lambda^{+}=\lambda^\omega$ and $(\lambda^{+})^M=(\lambda^\omega)^M$. Combining these expressions the equality $\lambda^+=(\lambda^+)^M$ follows. Finally this have proved that $\square_\lambda$ holds in $V[G]$ contradicting the supercompactness of $\kappa$.
\end{proof}
The same argument as before actually proves something stronger: for each cardinal $\lambda<\kappa$ the notion of $\lambda$-$C^{(1)}$-supercompactness is incompatible with $\square_\theta$ holding at each $\theta\in E^\kappa_{\leq \lambda}$. 
\begin{prop}
Assume $\mathrm{GCH}$ holds. Let $\kappa$ be a supercompact cardinal, $\lambda<\kappa$ and assume that for each $\theta\in E^\kappa_{\leq\lambda}$, $\square_\theta$-holds. Then there is no elementary embedding $j:V\rightarrow M$ such that $\crit(j)=\kappa$, $M^\lambda\subseteq M$ and $j(\kappa)$ being a limit cardinal.
\end{prop}
We will finish this section with the proof of corollary \ref{CorConsist1}:
\begin{proof}[Proof of Corollary \ref{CorConsist1}]
Let $V$ be a model of $\GCH$ with two $C^{(1)}$-supercompact cardinals $\kappa<\lambda$. The previous theorem shows that $V^{\mathbb{P}}$ is a model where $\kappa$ is no longer $C^{(1)}$-supercompact and in fact it is the first supercompact. Since $\mathbb{P}$ is a small forcing, $\lambda$ is still $C^{(1)}$-supercompact in $V^\mathbb{P}$ and greater than $\kappa$. Combining both things we get a model for the theory $$\text{``$\ZFC+\GCH+\min\mathfrak{S}<\min\mathfrak{S}^{(1)}$''}.$$
\end{proof}

\subsection{Proof of theorem \ref{StrengthTheorem1}}
The way we have proceed to make the first supercompact cardinal smaller than the first $C^{(1)}$-supercompact is very aggressive: namely, we have forced that scenario paying the prize of making the first supercompact to be the first ($\omega_1$-)strongly compact. Therefore it is natural to ask whether these three notions may be forced to be different. Recall that $\mathfrak{M}$, $\mathfrak{K}$, $\mathfrak{S}$ and $\mathfrak{S}^{(1)}$ stand for the class of measurable, strongly compact, supercompact and $C^{(1)}$-supercompact cardinals, respectively. In the next pages we shall present some modifications to the arguments of section 2.1 that will yield to a proof for the consistency of
$\text{``$\min\mathfrak{M}<\min\mathfrak{S}<\min\mathfrak{S}<\min\mathfrak{S}^{(1)}$''}$.

Assume $\mathrm{GCH}$ and let $\lambda<\kappa$ be two supercompact cardinals with a $C^{(1)}$-supercompact cardinal $\mu$ above $\kappa$. By virtue of a result of Apter \cite{ApterIndestr}, after a preparatory iteration $\mathbb{Q}\subseteq V_\lambda$ of length $\lambda$, one can assume that $\lambda$ is the first strongly compact and the first strong cardinal and besides it is indestructible by $<\lambda$-directed closed forcings (i.e. $\theta$-directed closed, all $\theta<\lambda$) which are also $\lambda$-strategically closed. Thereby in $V^\mathbb{Q}$ the $\mathrm{GCH}$ pattern above $\lambda$ is preserved, $\lambda$ is the first strongly compact but not the first measurable cardinal and $\kappa, \mu$ remain supercompact and $C^{(1)}$-supercompact, respectively. For the ease of notation henceforth we will assume that $V=V^\mathbb{Q}$. Analogously to the former section here we will add many $\square_{\theta,\eta}$-square sequences below $\kappa$ taking care that both the strong compactness of $\lambda$ and the supercompactness of $\kappa$ are preserved. The next forcing notion is discussed with full details in \cite[Section 9]{CummingsSquare} and it is the main ingredient of our argument:
\begin{defi}
Let $\theta$ be a singular cardinal and let $\langle\theta_i:\,i\in \cf\theta\rangle$ be an increasing and cofinal sequence in $\theta$ with $\theta_0> \cf\theta$. We will denote by $\mathbb{S}_\theta$ the forcing whose conditions are of the form
\[p=\langle C^p_{\alpha, i}:\,\lim(\alpha),\, \alpha\leq \gamma^p,\, i^p(\alpha)\leq i < \cf\theta\rangle\]
witnessing 
\begin{enumerate}
\item $\gamma^p$ is a limit ordinal less than $\theta^+$.
\item $i^p$ is a function such that $i^p(\alpha)<\cf \theta$ for each limit $\alpha<\gamma$.
\item If $i^p(\alpha)\leq i<\mu$ then $C^p_{\alpha, i}$ is a club in $\alpha$ of $\otp(C^p_{\alpha,i})<\theta_i$.
\item If $i^p(\alpha)\leq i<j<\mu$ then $C^p_{\alpha,i}\subseteq C^p_{\alpha, j}$.
\item If $i^p(\beta)\leq i<\mu$ and $\alpha\in \lim (C^p_{\beta, i})$ then $i(\alpha)\leq i$ and $C^p_{\alpha,i}=C^p_{\beta, i}\cap \alpha$.
\item If $\alpha$ and $\beta$ are limit ordinals with $\alpha<\beta\leq\gamma$ then there is some $i(\alpha)\leq i_0$ such that for every $i_0\leq i<\cf\theta$ then $\alpha\in\lim (C^p_{\beta, i})$.
\end{enumerate}
We will say that $p\leq q$ iff
\begin{enumerate}
\item[(a)] $\gamma^q\leq \gamma^p$.
\item[(b)] If $\alpha\leq \gamma^q$ then $i^q(\alpha)=i^p(\alpha)$ and $C^q_{\alpha,i}=C^p_{\alpha,i}$ for each $i^q(\alpha)\leq i< \cf \theta$.
\end{enumerate}
\end{defi}
It is illustrative to think on the conditions of $\mathbb{S}_\theta$ as matrices of clubs which are promises for a potential $\square_{\theta, \cf(\theta)}$-sequence. 
This forcing, besides of adding a $\square_{\theta,\cf\theta}$-sequence, is $\cf \theta$-directed and $<\theta$-strategically closed. The interested reader may find a detailed proof of both properties in \cite[Section 9]{CummingsSquare}.
 Since $\theta$ is singular, hence $\mathbb{S}_\theta$ does not add $\theta$-sequences, cardinals and cofinalities up to $\theta^+$ are preserved. 
Furthermore, as $\GCH$ holds above $\lambda$, for any singular cardinal $\theta>\lambda$ the forcing $\mathbb{S}_\theta$ has cardinality $\theta^{+}$ and thus preserves  all possibles cofinalities as well as the $\mathrm{GCH}$ pattern above $\lambda$. Without loss of generality we will make the assumption that all the cardinals in $\dom(\ell)$ are strong limit above $\lambda$ that are closed under $\ell$.
\begin{defi}
Let $\mathbb{P}^\ell_\kappa$ be the $\kappa$-Easton support iteration where  $\mathbb{P}^\ell_0$ is the trivial forcing and for each ordinal $\theta<\kappa$, if $\theta\in dom(\ell)\cap E^\kappa_\lambda$ and $\Vdash_{\mathbb{P}^\ell_\theta}\text{``$\check{\theta}^+$ is a cardinal''}$ then $\Vdash_{\mathbb{P}^\ell_\theta}\text{``$\mathbb{\dot{Q}}_\theta=\mathbb{S}_\theta$''}$ and $\Vdash_{\mathbb{P}^\ell_\theta}\text{`` $\mathbb{\dot{Q}}_\theta$ is trivial''}$, otherwise. 
\end{defi}
The iteration $\mathbb{P}^\ell_\kappa$ is clearly $<\lambda$-directed closed and $\lambda$-strategically closed and thus $\lambda$ remains strongly compact and strong in the generic extension. The next proposition is the corresponding version  of proposition \ref{PreservationOfSuper} in the current setting:
\begin{prop}\label{NotC1Further}
The following statements are true in $V^{\mathbb{P}^\ell_\kappa}$:
\begin{enumerate}
\item $\lambda$ is strongly compact and strong and $\mu$ is $C^{(1)}$-supercompact.
\item There is a stationary set $S^*\subseteq E^\kappa_\lambda$ such that for every $\theta\in S^*$, $\square_{\theta,\lambda}$ holds. In particular, there is no strongly compact between $\lambda$ and $\kappa$.
\item $\kappa$ is supercompact but not $C^{(1)}$-supercompact. In fact, there is no elementary emebedding $j: V^{\mathbb{P}^\ell_\kappa}\rightarrow M$ with $j(\kappa)$ being a limit cardinal and $M^\lambda\subseteq M$. In particular, $\kappa$ is the first supercompact cardinal.
\end{enumerate}
\end{prop}
\begin{proof}$ $
\begin{enumerate}
\item It follows from Apter's result.
\item Let any $\theta\in \mathrm{dom}(\ell)\cap E^\kappa_\lambda$ and notice that $\Vdash_{\mathbb{P}^\ell_{\theta+1}} \text{`` $\square_{\theta,\lambda}$ holds''}$. Set $\theta^*$ be the least cardinal in $\mathrm{dom}(\ell)\,\cap\, E^\kappa_\lambda$ above $\theta$. The iteration restricted to the interval $[\theta^*, \kappa)$ is $\theta^*$-startegically closed, hence $(\theta^+)^{V^{\mathbb{P}^\ell_{\theta+1}}}$ is preserved, and thus $\Vdash_{\mathbb{P}^\ell_\kappa} \text{``$\square_{\theta,\lambda}$ holds''}$. Finally, the iteration $\mathbb{P}^\ell_\kappa$ is $\kappa$-cc because $\kappa$ is Mahlo and thus the set $\mathrm{dom}(\ell)\cap E^\kappa_\lambda$ remains stationary in the generic extension.

 The further claim is a consequence of a well-known argument due to Solovay that we exhibit only for completeness. Aiming for a contradiction suppose that there is some $\lambda< \eta< \kappa$ being $\theta^+$-strongly compact cardinal, some $\theta\geq\eta$ in $S^*$. Let $j: V\rightarrow M$ be an elementary embedding with $cp(j)=\eta$ and $D\in M$ such that $j''\theta^+\subseteq D$ and $M\models |D|<j(\eta)$. Let $\vec{C}=\langle \mathcal{C}_\alpha:\,\lim(\alpha), \alpha\in\theta^+\rangle$ be the $\square_{\theta, \lambda}$-sequence forced by $\mathbb{P}^\ell_\kappa$ and $\vec{D}=j(\vec{C})$. Set $\theta^*=\sup(j''\theta^+)$ and notice that $\theta^*<j(\theta)^+$ and $\cf^M(\theta^*)<j(\eta)$. Let $D_{\theta^*}\in \mathcal{D}_{\theta^*}$ and define \linebreak$C=\{\alpha\in\theta^+:\, j(\alpha)\in D_{\theta^*}\}$ the asso\-ciated $<\eta$-club. Let $\gamma>\theta$ be a limit point of $C$ with $cof(\gamma)=\omega$ and $|C\cap \gamma|=\theta$. By continuity of $j$ in $\gamma$ it is the case that\linebreak $j(\gamma)\in \lim(D_{\theta^*})$. Notice that for every $\alpha\in C\cap \gamma$, the formula $\varphi(\alpha,\gamma)$ 
\begin{eqnarray*}
``\exists\theta'\in j(\theta)^+\, \exists D_{\theta'}\in \vec{D}(\theta')\, (cof(\theta')<j(\eta)\,\\ \wedge\, j(\gamma)\in\lim(D_{\theta'})  \wedge\,j(\alpha)\in D_{\theta'}\cap j(\gamma)) ''
\end{eqnarray*}
is true un in $M$ as witnessed by $\theta^*$. Thus for each $\alpha\in C\cap \gamma$ there is some $C_{\theta_\alpha}$ such that $\cf(\theta_\alpha)<\eta$, $\gamma\in lim(C_{\theta_\alpha})$ and $\alpha\in C_{\theta_\alpha}\cap \gamma$. Notice that all of these $C_{\theta_\alpha}\cap \gamma$ lie in $\mathcal{C}_\gamma$ and have cardinality less than $\theta$ (since $\cf(\theta_\alpha)<\eta<\theta$). Thus $C\cap \gamma$ can be covered by the union of all clubs in $\mathcal{C}_\gamma$ with cardinality less than $\theta$. Since $|\mathcal{C}_\gamma|\leq \cf(\theta)<\theta$, this union has cardinality less than $\theta$. Contradiction. 
\item The argument is the same as in proposition \ref{PreservationOfSuper} and theorem \ref{supercompnotc1} noting that $\mathbb{P}^\ell_\kappa$ preserve the $\GCH$ pattern above $\lambda$. 
\end{enumerate}
\end{proof}
We can also say something else about the status of $\lambda$ in the generic extension $V^{\mathbb{P}^\ell_\kappa}$:
\begin{prop}\label{InVPell}
The cardinal $\lambda$ is the first strong cardinal and the first strongly compact in $V^{\mathbb{P}^{\ell}_\kappa}$. In particular, $\lambda$ is greater than the first measurable of $V^{\mathbb{P}^\ell_\kappa}$.
\end{prop}
\begin{proof}
Let us simply show that $\lambda$ is still the least strong in the generic extension since the claim about strong compactness can be proved similarly. Let $\lambda^\star<\lambda$ be a strong cardinal in $V^{\mathbb{P}^{\ell}_\kappa}$. The property of being  a strong cardinal is $\Pi_2$ definable and any strong cardinal is a $C^{(2)}$-cardinal. It then follows from this and from (1) of proposition \ref{NotC1Further} that $\lambda^\star$ is strong within $V[G]_\lambda$. On the other hand notice that $V[G]_\lambda=V_\lambda$ because the iteration is $\lambda$-distributive, hence $\lambda^\star$ is strong in $V_\lambda$. Finally since $\lambda$ was a strong cardinal in $V$, hence $C^{(2)}$, it is the case that $\lambda^\star$ is also a strong cardinal in $V$ below $\lambda$. This yields to contradiction with the minimality of $\lambda$ in $V$.
\end{proof}
Combining propositions \ref{NotC1Further} and \ref{InVPell} the claim of theorem \ref{StrengthTheorem1} and corollary \ref{CorConsist2} easily follows.

\section{Identity crises: the first $C^{(n)}$-supercompact can be the first strongly compact. }\label{Section 3}

Let $\mathscr{L}$ be a large cardinal property and $\kappa$ be a cardinal such that $\mathscr{L}(\kappa)$. We say that $\ell:\kappa\rightarrow\kappa$ is a \textit{$\mathscr{L}$-fast function} on $\kappa$ if for every $\lambda>\kappa$ there is an $\mathscr{L}$-elementary embedding $j:V\rightarrow M$ with $\crit(j)=\kappa$ and $j(\ell)(\kappa)>\lambda$. There are many typical example of such sort of functions among which one must to highlight the Laver functions (see \cite{Lav}). Under our convention a Laver function on a supercompact cardinals is the same as a $\mathfrak{S}$-fast function. Another natural example of this sort of objects is given by Cohen reals where the homogeneity of $\Add(\kappa,1)$ yields to the desired fast behaviour (see e.g. lemma \ref{LemmaCohenfunction}). 

By results of Tsaprounis \cite{Tsan} it is known that any \Cn-extendible cardinal carries a $\mathfrak{E}^{(n)}$-Laver function and moreover that the standard Jensen iteration to force global $\mathrm{GCH}$ preserves \Cn-extendibility. For a general version of Tsaprounis' theorem see \cite{BP}.



Since the discovering of Laver functions fast functions have played a central role in iteration arguments. Essentially this sort of functions allows us to find arbitrary segments of $j(\mathbb{P})$ where the iteration is trivial which is a crucial property for lifting elementary embeddings.

Regrettably, due to the general lack of understanding of \Cn-supercompact cardinals, anything is known about the existence $\mathfrak{S}^{(n)}$-fast functions. The naive strategy for proving they exist will lead us to mimic Laver's construction of a Laver function even though we will eventually realize that this does not work. More precisely, there are obstacles to reflect the formula asserting that there is a counterexample for the existence of a $\mathfrak{S}^{(n)}$-fast function since it is  $\Pi_{n+2}$ while \Cn-supercompact cardinals are not necessarily $C^{(n+2)}$-correct\footnote{This consequence of theorem \ref{MainTheorem}}.


An alternative strategy is to discuss whether some forcing notion adding a $\mathfrak{S}^{(n)}$-fast function preserves \Cn-supercompact cardinals. Nevertheless this strategy turns to be very problematic as we shall argue in Section \ref{OpenQuestions}. Anyway if we are given a \Cn-supercompact cardinal $\kappa$ in a generic extension $V[\ell]$ with $\ell\subseteq\kappa$ being a Cohen real (e.g.\ as in Tsaprounis's theorem), we may assume that $\ell$ is a $\mathfrak{S}^{(n)}$-fast function in $V[\ell]$ by virtue of the next result:
\begin{lemma}\label{LemmaCohenfunction}
Let $\ell:\kappa\rightarrow \kappa$ be a Cohen function over $V$. Let $j\colon V\to M$ be an elementary embedding with critical point $\kappa$ and let $\lambda < j(\kappa)$.  If there is an extension of $j$ to an elementary embedding $\tilde{j} \colon V[\ell] \to M[\tilde\ell]$ with $\tilde{j}(\ell) = \tilde{\ell}$ that extends $j$ then there is another extension of $j$, $\tilde{j}'\colon V[\ell] \to M[\tilde{\ell}']$, such that $\tilde{j}'(\ell) = \tilde{\ell}'$ and $\tilde{\ell}'(\kappa) = \lambda$. Moreover, if $\tilde{j}$ witness that $\kappa$ be a $\lambda$-\Cn-supercompact cardinal in $V[\ell]$ then so does $\tilde{j}'$. 
\end{lemma}
\begin{proof}
Let $p=\{\langle \kappa,\lambda\rangle\}\in \Add(j(\kappa),1)^M$. Since the Cohen forcing is homogeneous, one can find a $M$-generic filter $H$ for the forcing $\Add(j(\kappa),1)^M$ such that $p\in H$, $\bigcup H \restriction \kappa = \tilde{\ell}\restriction \kappa$ and $M[\tilde{\ell}]=M[H]$. By the elementarity of $\tilde{j}$, $\tilde{\ell}\restriction \kappa = \tilde{j}(\ell)\restriction \kappa = \ell$. Let $\tilde{\ell}' = \bigcup H$. By Silver's argument, $j$ extends to an elementary embedding $\tilde{j}' \colon V[\ell] \to M[\tilde{\ell}'] = M[\tilde{\ell}]$. If $\tilde{j}$ was a $\lambda$-\Cn-supercompact embedding then so is $\tilde{j}'$ since $M[\tilde{\ell}] = M[\tilde{\ell}']$ and $\tilde{j}(\kappa) = \tilde{j}'(\kappa) = j(\kappa)$. 
\end{proof}

All the issues described so far can be framed within the setting of preservation of \Cn-supercompactness by forcing. Broadly speaking, the main obstacle for developing a general theory of preservation for \Cn-supercompact cardinals is the disagreement between the strong correctness of $j(\kappa)$ and the little resemblance between $M$ and the universe. More precisely \Cn-supercompact embeddings may not be superstrong and thus this opens the door to have target models $M$ that are not more correct than $\Sigma_2$-correct (i.e $M\prec_2 V$) regardless $j(\kappa)\in C^{(n)}$. At Section 5 we will cover this problematic with all details.


\subsection{Magidor Product}
Henceforth we will assume that $n\geq 1$, $\kappa$ is a \Cn-supercompact cardinal and $\ell:\kappa\rightarrow\kappa$ is a $\mathfrak{S}^{(n)}$-fast function with $\ran (\ell)=\langle \kappa_\alpha:\, \alpha<\kappa\rangle $ a set of measurable cardinals which does not contain their limit points; namely, for every $\alpha<\kappa$, $\sup_{\beta<\alpha} \kappa_\beta<\kappa_\alpha$. 
\begin{defi}[\textbf{Magidor product}]\label{MagProd}
Let $\kappa$ be a regular cardinal and $A=\langle \kappa_\alpha:\,\alpha<\kappa\rangle$ be a subset of measurable cardinals below $\kappa$ which does not contain their limit points. Set $U_\alpha$ be a normal measure on $\kappa_\alpha$, each $\alpha<\kappa$. The $\kappa$-Magidor product with respect to $A$, $\mathbb{M}_{A,\kappa}$,  is the set of all sequences $p=\langle\langle s(\alpha), A_\alpha\rangle:\, \alpha<\kappa \rangle$ such that
\begin{itemize}
\item[(a)] For every $\alpha<\kappa$, $(s(\alpha), A_\alpha)\in\mathbb{P}_{U_\alpha}$, where $\mathbb{P}_{U_\alpha}$ stands for the Prikry forcing with respect to the normal measure $U_\alpha$.
\item[(b)] $\{\alpha<\kappa: \, s(\alpha)\neq \emptyset\}\in [\kappa]^{<\aleph_0}$.
\end{itemize} 
Given two conditions $p,q\in\mathbb{M}_{A,\kappa}$, $p\leq q$ ($p$ is stronger than $q$) if for every $\alpha<\kappa$, $p(\alpha)\leq_{\mathbb{P}_{U_\alpha}}q(\alpha)$. We will also say that $p$ is a direct extension of $q$, $p\leq^\star q$ if for every $\alpha<\kappa$, $p(\alpha)\leq^\star_{\mathbb{P}_{U_\alpha}}q(\alpha)$
\end{defi}
It is illustrative to think on $\mathbb{M}_{A,\kappa}$ as a particular case of a Magidor iteration of Prikry forcings as presented in definition 6.1 of \cite{GitPrikry}. Specifically, provided that $A$ does not contain their limit points, one can easily check that $\mathbb{M}_{A,\kappa}$ is isomorphic to the Magidor iteration of Prikry forcings at each $\kappa_\alpha\in A$ below the condition $\langle \langle \emptyset, \kappa_\alpha\rangle:\,\alpha\in\kappa\rangle$.

On the sequel we shall adopt the notation $\mathbb{M}$ instead of the cumbersome $\mathbb{M}_{\ran (\ell), \kappa}$ as long as the set $A$ and the cardinal $\kappa$ are clear from the context. Our main aim along this section is to prove that $\mathbb{M}$ preserves \Cn-supercompactness of $\kappa$ lifting the corresponding ground model embeddings to \Cn-supercompact embeddings in the generic extension. As we shall argue in such generic extension the first \Cn-supercompact cardinal coincides with the first ($\omega_1$-)strongly compact cardinal.

The key point to carry out the lifting arguments is that the generics of $\mathbb{M}$ are not arbitrary objects but are essentially given by sequences of generics for the corresponding Prikry forcings. It is widely known that Mathias criteria of genericity (see e.g. \cite{GitPrikry}) implies that the critical sequence $\langle \theta_n:\,n\in\omega\rangle$ of a $\omega$-length iteration of ultrapowers with respect to some measure over $\kappa$ defines a Prikry generic $C\in V$ for $\mathbb{P}_{U_\omega}$ over $M_\omega$\footnote{Here $U_\omega$ is the measure over $\kappa_\omega=\sup_n\kappa_n$ generated by the family of sets $\{A_n:n\in\omega\}$, where $A_n=\{\kappa_m:m<n\}$. }
Therefore, combining both things, iterated ultrapowers seems to provide a standard tool to define generic filters for $\mathbb{M}$ and thus it turns to be necessary to prove a similar version to the Mathias criteria for $\mathbb{M}$. In the next section we shall prove that $\mathbb{M}$ enjoys certain property also satisfied by the Prikry forcing  that constitutes the main ingredient for the proof of Mathias criteria of generecity. We have called this property \textit{Mathias-Prikry property}:
\begin{lemma}
Let $\mathbb{P}$ be the Prikry forcing with respect to some normal measure $U$. Then $\mathbb{P}$ enjoys the Mathias-Prikry property; namely, for every condition $\langle s,A\rangle\in\mathbb{P}$ and every dense open set $D\subseteq\mathbb{P}$ there are $n_s\in\omega$ and $A\in U$ such that for every $m\geq n_p$ and every $t\in[A]^m$, $\langle s\frown t, A\setminus\max(t)+1\rangle\in D$.
\end{lemma}
\begin{proof}
See (see lemma 1.13 of \cite{GitPrikry}). 
\end{proof}
Once we prove that $\mathbb{M}$ enjoys the Mathias-Prikry property the sketch for the construction of the generics will be the following. Let  $j: V\rightarrow M$ be a $\lambda$-\Cn-supercompact embedding, $A^\star=\langle {\kappa}_\alpha: \lambda<\alpha<j(\kappa)\rangle$ be a family of $M$-measurable cardinals not containing their limit points and $U^\star = \langle \tilde{U}_\alpha : \lambda < \alpha < j(\kappa)\rangle$ be a sequence of measures over ${\kappa}_{\alpha}$. Define over $M$ a $\omega\cdot j(\kappa)$-iteration of ultrapowers $\langle M_\alpha, j_{\alpha, \beta} \mid \alpha \leq \beta \leq \omega\cdot \mu\rangle$ where each ${\kappa}_\alpha$ is iterated $\omega$-many times. By previous comments this iteration yields to a family of ($M$-definable)  generic filters $\langle H_\alpha:\,\lambda<\alpha<j(\kappa)\rangle$ for each Prikry forcing $\mathbb{P}_{U_{\omega\cdot\alpha}}$ which defines -here is where the Mathias-Prikry property comes into play- a $\mathbb{M}^{M_{\omega\cdot \mu}}$-generic filter over $M_{\omega\cdot\mu}$. We will finally show that the embedding $j_{0,\omega\cdot} \circ j$ lifts to a $\lambda$-\Cn-supercompact embedding in $V^\mathbb{M}$ thus proving that $\kappa$ remains \Cn-supercompact in the generic extension.

\subsection{$\mathbb{M}$ and the Mathias-Prikry property}
\begin{defi}
A function $s\in \prod_{\alpha\in\kappa} \kappa_\alpha^{<\omega}$ is a \emph{stem} if $s(\alpha)$ is a strictly increasing sequence of cardinals and $\{\alpha<\kappa:\, s(\alpha) \neq \emptyset\}\in [\kappa]^{<\aleph_0}$.
Let $St$ be the set of all stems. For $s\in St$, we let the \textit{support of $s$}, $\supp s$,  be an increasing enumeration $\langle\alpha_i:i\leq n\rangle$ of the non trivial coordinates of $s$. The length sequence of a stem $s$ is $\len s=\langle \len s(\alpha):\,\alpha<\kappa\rangle$.
\end{defi}
Notice that a length sequence $\len s$ completely determines $\supp s$. Thus, all the relevant information (i.e.\ the support and the lengths of the corresponding sequences) about a stem $s$ is encoded within $\len s$. Let $ \bigoplus_{\alpha < \kappa} \omega$ denote the set of all $\kappa$-sequences of natural numbers which are non-zero only in a finite set. Let $\vec{\gamma}\in \bigoplus_{\alpha < \kappa} \omega$, we will set $\vec{\gamma}_{\neq 0}=\{\alpha\in\kappa:\,\vec{\gamma}(\alpha)\neq 0\}$. If $\vec{\gamma},\vec{\gamma}'\in \bigoplus_{\alpha < \kappa} \omega$ we will write $\vec{\gamma}\leq_p \vec{\gamma}'$ if for every $\alpha<\kappa$, $\vec{\gamma}(\alpha)\leq \vec{\gamma}'(\alpha)$.
\begin{lemma}[Finite Diagonal Intersection]\label{lemma: finite diagonal intersection}
Let $\vec{\gamma}$ be a length sequence and $\langle B^s_\alpha \mid \alpha < \kappa, s\in St, \len s = \vec{\gamma}\rangle$, $B^s_\alpha\in U_\alpha$ with $\min B^s_\alpha \geq \max s(\alpha)$. There is a sequence of large sets $\langle C_\alpha \mid \alpha < \kappa\rangle$ such that for every stem $s \in \prod_{\alpha\in\kappa} C_\alpha^{<\omega}$, $\len s=\vec{\gamma}$, $B_\alpha^s \supseteq C_\alpha \setminus \max s(\alpha)$ for all $\alpha$. 
\end{lemma}

\begin{proof}
Let us show that the theorem holds by induction over the amount of non-zero coordinates of $\vec{\gamma}$, $|\vec{\gamma}_{\neq 0}|$. Suppose that $|\vec{\gamma}_{\neq 0}| =0$, then there is only one stem with this support (namely the 0 function). Thus defining $C_\alpha=B^s_\alpha$, we are done. Now suppose by induction that for every length sequence $\vec{\gamma}'$ with $|\vec{\gamma}'_{\neq 0}|\leq n$, for every ordinal $\delta$ and every family of large sets $\langle B^s_\alpha:\alpha<\delta,\, s\in St,\, \supp s=\vec{\gamma}'\rangle$ there is $\langle C_\alpha:\,\alpha<\delta\rangle$ witnessing the theorem. 

Let $\vec{\gamma}$ be a length sequence with $|\vec{\gamma}_{\neq 0}|=n+1$ and $\langle B^s_\alpha:\alpha<\kappa,\, s\in St,\, \len s=\vec{\gamma}\rangle$ be a family of large sets. Say that $\max(\vec{\gamma}_{\neq 0})=\delta$. Notice that there are at most $\kappa_\delta$-many stems with such support and thus for every $\delta<\alpha$ the set $C_\alpha=\bigcap_{s\in St,\len s=\vec{\gamma}} B^s_\alpha$ is an element of $U_\alpha$. Let us work now with the truncated family $\langle B^s_\alpha:\alpha< \delta,\, s\in St,\, \len s=\vec{\gamma}\rangle$. All the stems with length sequence $\len s=\vec{\gamma}$ are built by some $s'\in St$ with $\len s'=\vec{\gamma}\upharpoonright n$ and some $\vec{\eta}\in \kappa_{\delta}^{\vec{\gamma}(\delta)}$. Namely, $s=s'\frown \vec{\eta}$.\footnote{This denotes the stem $s$ which is equal to $s'$ on all coordinates except in $\delta$, in which is  equal to the sequence $\vec{\eta}$} For each possible extension $\vec{\eta}$, one has a family $\mathcal{B}_{\vec{\eta}}= \langle B^{s'}_\alpha(\vec{\eta}):\alpha< \delta,\, s'\in St,\, \len s'=\vec{\gamma}\upharpoonright n\rangle$ of large sets.  By the discreteness of the measurables, there is a large set $A_{\delta}\in U_{\delta}$ such that the families $\mathcal{B}_{\vec{\eta}}$ are the same for every $\vec{\eta}\in A_{\delta}^{\vec{\gamma}(\delta)}$. Let $\langle B^{s'}_\alpha:\alpha< \delta,\, s'\in St,\, \len s'=\vec{\gamma}\upharpoonright n\rangle$  be this common value and apply the induction hypothesis to obtain a family $\langle C_\alpha:\,\alpha<\delta\rangle$ witnessing the theorem. For the coordinate $\delta$ define $C_{\delta}=A_{\delta}\cap \bigtriangleup\{B^s_{\delta}: s\in St,\,\len s=\vec{\gamma}\}$ where $\bigtriangleup\{B^s_{\delta}: s\in St,\,\len s=\vec{\gamma}\}$ is defined as:
$$\{\beta\in \kappa_{\delta}:\,(s\in St\,\wedge\,\len s=\vec{\gamma}\,\wedge\, \max(s(\delta))<\beta)\rightarrow\beta\in B^s_{\delta}\}.$$
It is routine to check that the family $\langle C_\alpha:\,\alpha<\kappa\rangle$ witnesses the theorem for the support $\vec{\gamma}$.
\end{proof}
\begin{lemma}[R\"{o}wbottom Lemma]\label{Rowbottom}
Let $f\colon St \to 2$ be a function. There is a sequence of large sets  $\langle C_\alpha \mid \alpha < \kappa\rangle$ and a function $g\colon \bigoplus_{\alpha < \kappa}\omega\rightarrow 2  $  such that for every stem $s\in \prod_{\alpha\in\kappa} C_\alpha^{<\omega}$, $f(s) = g(\supp s)$.  
\end{lemma}
\begin{proof}
Fix $\alpha\in \kappa$ an let $St_\alpha=\{s\in St:\, \max(\supp s)=\alpha\}$. We are going to define by induction over $n\in\omega$ a sequence of functions $f_n\upharpoonright\alpha: St_\alpha \rightarrow 2$ and a sequence of $U_\alpha$-large sets $\langle A_{\alpha,n}:\,n\in\omega\rangle$. Let $f_0=f$ and $A_{\alpha,0}=\kappa_\alpha\setminus\{0\}$ and let us show how to proceed on larger $n$'s. Denote by $St_{\alpha,n}$ the set of all stems such that $\alpha=\max(\supp s)$ and $s(\alpha)\in A_{\alpha, n}^{<\omega}$. For each $s\in St_{\alpha,n}$, consider $F^n_{\alpha,s\upharpoonright\alpha}: \kappa_\alpha^{<\omega}\rightarrow 2$ defined by $\vec{\eta}\mapsto f_n(s^\alpha_{\vec{\eta}})$. Here $s^\alpha_{\vec{\eta}}$ stands for the stem $s^*$ which coincides with $s$ in all the coordinates except in $\alpha$ where it is $\vec{\eta}$. By R\"owbottom theorem one can find a homogeneous set $H_{s\upharpoonright\alpha,\alpha}\subseteq A_{\alpha, n}$ for this function. Define $A_{\alpha, n+1}=\bigcap\{H_{s\upharpoonright\alpha,\alpha}:\, s\in St_{\alpha,n}\}$ and notice that this is a $U_\alpha$-large set because this intersection runs for less than $\kappa_\alpha$-many sets. For each $s\in St_{\alpha,n}$, with $s(\alpha)\in A_{\alpha, n+1}^{<\omega}$, define $f_{n+1}(s^\alpha_{\langle 0\rangle^{s(\alpha)}})=f_n(s)$. Here $\langle 0\rangle^{s(\alpha)}$ stands for the sequence of length $\len s(\alpha)$ of $0$'s\footnote{By convention, $\langle 0\rangle^0=\emptyset$.}. This finishes the induction over $n$.

Repeating the above argument for each $\alpha<\kappa$, one gets a sequence of large sets $\langle A_{\alpha, n}:\,\alpha<\kappa, n\in\omega\rangle$ and a sequence of functions $\langle f_n:\,n\in \omega\rangle$. Let $C_\alpha=\bigcap_{n\in\omega} A_{\alpha, n}$ and $St^\star_\alpha=\bigcap_{n\in\omega} St_{\alpha,n}$ and notice that
\begin{equation}\label{equation}
\forall\alpha\in \kappa\,\forall s\in St^\star_\alpha\,\forall n\in\omega\,(f_{n+1}(s^\alpha_{\langle 0\rangle^{s(\alpha)}})=f_n(s)).
\end{equation}
For every $m\in\omega$, we will prove by induction that for every stem $s\in \prod_{\alpha\in\kappa} (C_\alpha\cup \{0\})^{<\omega}$, $f_m(s)=f_{m+n}(r)$, where $n=|\{\alpha:\, s(\alpha)\neq \langle 0\rangle^{s(\alpha)}\}|$ and $r$ is such that $r(\alpha)=\langle 0\rangle^{s(\alpha)}$ if $\alpha\in \supp s$ and $r(\alpha)=\emptyset$, otherwise. The induction runs over this $n$'s.

Let $m\in\omega$ be fixed. Let us prove for the sake of clarity the first two inductive steps. If $s$ is a stem such that $|\{\alpha:\, s(\alpha)\neq \langle 0\rangle^{s(\alpha)}\}|=0$ then the claim is true since $r=s$. On the other hand, if $\{\alpha:\, s(\alpha)\neq \langle 0\rangle^{s(\alpha)}\}=\{\beta\}$, then $s\in St^\star_{\beta}$ and thus $f_m(s)=f_{m+1}(s^\beta_{\langle 0\rangle^{s(\beta)}})$ by the equation \eqref{equation}. Notice that $s^\beta_{\langle 0\rangle^{s(\beta)}}=r$, and we are done.

Suppose that the claim is true for stems $s$ such that $|\{\alpha:\, s(\alpha)\neq \langle 0\rangle^{s(\alpha)}\}|=n$. Let $s\in \prod_{\alpha\in\kappa} (C_\alpha\cup \{0\})^{<\omega}$ with $|\{\alpha:\, s(\alpha)\neq \langle 0\rangle^{s(\alpha)}\}|=n+1$ such that $s\in St^\star_\alpha$, for some ordinal $\alpha$. By equation \ref{equation}, $f_m(s)=f_{m+1}(s^\alpha_{\langle0\rangle^{s(\alpha)}})$. Now $s^*=s^\alpha_{\langle 0\rangle^{s(\alpha)}}$ is such that $|\{\beta:\, s(\beta)\neq \langle 0\rangle^{s^*(\beta)}\}|=n$ so by induction  we know that $f_{m+1}(s^\alpha_{\langle0\rangle^{s(\alpha)}})=f_{m+n+1} (r)$, where $r(\beta)=\langle 0\rangle^{s(\beta)}$ for every $\beta\in \supp s$. This shows that $f_m(s)=f_{m+n+1}(r)$. In particular, for each $s\in\prod_{\alpha\in\kappa} C_\alpha^{<\omega}$, $f(s)=f_n(r)$, where $n=| \supp s|$. Thus defining $g(\supp s)=f_{|\supp s|}(r)$, we are done.
\end{proof}
Both lemmas yields to the proof of the \emph{Mathias-Prikry Property} for $\mathbb{M}$.
\begin{lemma}[Mathias-Prikry Property]\label{StrongPrikry}
Let $D$ be a dense open subset of $\mathbb{M}$ and $p\in\mathbb{M}$. There is a direct extension $p^\star \leq^\star p$ and some $\vec{\gamma}$ which is a length sequence of a stem, such that for all $q \leq p^\star$ with stem $s_q$ and $\vec{\gamma}\leq_p \len s_q$ then $q\in D$.
\end{lemma}
\begin{proof}
Let $s\in St$ be the stem of $p$. Let $f_s \colon St \to 2$ be the function that sends a stem $t$ to $1$ if the concatenation of both stems $s\frown t$ is an stem and there is a sequence of large sets $\langle B^{s\frown t}:\, \alpha<\kappa\rangle$ such that the resulting condition is in $D$. Otherwise, define this value as $0$. Applying Lemma \ref{Rowbottom}, there is a sequence of large sets $\langle C_\alpha:\,\alpha< \kappa\rangle$ and a function $g\colon\bigoplus_{\alpha<\kappa}\omega\rightarrow 2$ such that for every $t\in\prod_{\alpha\in\kappa} C^{<\omega}_\alpha$, $f_s(t)=g(\supp t)$. Since $D$ is dense open it is clear that there is $t^*\in \prod_{\alpha\in\kappa} C^{<\omega}_\alpha$ such that $f_s(t^*)=1$. Thus if $\len t^*=\vec{\gamma}$ we have that $f_s(t)=1$, for every $t\in\prod C^{<\omega}_\alpha$ with $\len t=\vec{\gamma}$.
By definition, for every stem $t$ with length sequence $\vec{\gamma}$, there is a sequence of large sets $\langle B_\alpha^{s\frown t} \mid \alpha < \kappa\rangle$ such that the corresponding forcing condition lies in $D$. Apply Lemma \ref{lemma: finite diagonal intersection} to $\langle B_\alpha^{s\frown t} \mid \alpha < \kappa,\,\len t=\vec{\gamma}\rangle$ and let $\langle C^{'}_\alpha:\alpha<\kappa\rangle$ be the family of large sets witnessing it. For each $\alpha<\kappa$, define $C^*_\alpha=C^{'}_\alpha\cap C_\alpha$. Let $p^\star$ be the condition in $\mathbb{M}$ with stem $s$ and large sets $\langle C^{*}_\alpha:\,\alpha<\kappa\rangle$. If $q\leq p^\star$ and $\vec{\gamma}\leq_{p} \len s_q$, then $q$ is stronger than some condition with stem $s\frown t$ with $t\in \prod_{\alpha\in\kappa}{ (C^{*}_\alpha)}^{<\omega}$ and large sets $\langle B^{s\frown t}_\alpha:\,\alpha<\kappa\rangle$. By the above argument,  this condition is in $D$ and thus $q$ also.
\end{proof}
\subsection{Preserving \Cn-supercompactness}
Let $\lambda>\kappa$ and $j: V\rightarrow M$ be a $\lambda$-\Cn-supercompact embedding such that $j(\ell)(\kappa)>\lambda$. Recall that the existence of such embeddings are guaranteed by Lemma \ref{LemmaCohenfunction}.
\begin{lemma}\label{MainLemma}
Let $G\subseteq \mathbb{M}$ a $V$-generic filter. Then there is  an elementary embedding \linebreak$j^\star:V[G]\rightarrow M^\star[G\times H]$ witnessing the $\lambda$-\Cn-supercompactness of $\kappa$ in $V[G]$.
\end{lemma}
\begin{proof}
Recall that $\mathbb{M}=\mathbb{M}_{\ran (\ell), \kappa}$ so by elementarity $j(\mathbb{M})=\mathbb{M}^M_{\ran (j(\ell)), j(\kappa)}$. It is obvious that the forcing $j(\mathbb{M})$ factorizes as $\mathbb{M}\times j(\mathbb{M})/\mathbb{M}$ where $j(\mathbb{M})/\mathbb{M}$ is the $M$-version for the magidor product $\mathbb{M}_{(\ran (j(\ell))\setminus\ran \ell), j(\kappa)}$. Since we have taken $j$ in such a way that $j(\ell)(\kappa)>\lambda$, then $\ran (j(\ell))\setminus\ran \ell)$ can be written as an increasing sequence of measurable cardinals $\langle\kappa_\alpha:\,\alpha<j(\kappa)\rangle$ such that $\kappa_0>\lambda$ and that for every $\alpha<j(\kappa)$, $\sup_{\beta<\alpha} \kappa_\beta<\kappa_\alpha$. For ease of notation set $\mu=j(\kappa)$ and $\mathbb{M}^*= j(\mathbb{M})/\mathbb{M}$. Working in $M$ we shall build an iteration of ultrapowers $\langle M_\alpha, j_{\alpha, \beta} \mid \alpha \leq \beta \leq \omega\cdot \mu\rangle$ and we will show that $\langle C_\alpha \mid \alpha < \mu\rangle$ generates a $M_{\omega\cdot\mu}$-generic for the Magidor product $j_{\omega\mu}(\mathbb{M}^*)$, where $C_\alpha = \langle \rho_{\alpha}^n \mid n < \omega\rangle$ is the $\alpha$th-critical sequence of the iteration.  Let $M_0 = M$, $j_{00}=\mathrm{id}$ and $\vec{U}=\langle U_\alpha:\,\alpha\in\mu\rangle$. For limit $\alpha$, let $M_{\alpha}$ be the direct limit of the system $\langle M_{\beta}, j_{\beta,\gamma} \mid \beta \leq \gamma < \alpha\rangle$ while for successor cases we set
 \[M_{\omega \cdot \alpha + n + 1} = \Ult(M_{\omega \cdot \alpha + n}, j_{\omega \cdot \alpha + n}(\vec{U})_\alpha)\]
, each $n\in\omega$.
Let $j_{\omega \cdot \alpha + n, \omega \cdot \alpha + n + 1}$ be the corresponding ultrapower map and define $j_{\beta, \omega \cdot \alpha + n + 1}$, for $\beta<\omega\cdot\alpha+n+1$, in the only possible way: namely, \[\rho_\alpha^n = \crit j_{\omega \cdot \alpha + n, \omega \cdot \alpha + n + 1} = j_{\omega \cdot \alpha + n}(\kappa_\alpha).\]
Notice that $\rho^0_0>\lambda$ and moreover $\kappa_\alpha=\rho^0_\alpha>\alpha$ for every $\alpha<\mu$, by discreteness of the measurables.
 By standard computations of iterated ultrapowers one can show that $\bar{j}(\mu)=\mu$.
For the ease of notation, on the sequel we will write $\bar{j} = j_{\omega\cdot \mu},\, M^\star = M_{\omega\cdot\mu}$. Consider,
$$H=\{p\in \bar{j}(\mathbb{M}^*):\, \forall \alpha\in\mu\,\forall q\in H_\alpha\;  (p(\alpha)\parallel q)\}$$
where $H_\alpha=\{\langle s, A\rangle\in\mathbb{P}_{U_\alpha}:\,s\lhd C_\alpha,\; C_\alpha\setminus\max(s)\subseteq A\ \}$; i.e. the Prikry generic defined by the critical sequence $C_\alpha$. We claim that $H$ is a generic filter for the Magidor product $\bar{j}(\mathbb{M}^\star)$ over $M^\star$.

\begin{claim}\label{ClaimGenericity}
The filter $H$ is $M^\star$ generic for $\bar{j}(\mathbb{M}^*)$.
\end{claim}
\begin{proof}[Proof of claim]
Let $D \in M$ be a dense open subset of $\bar{j}(\mathbb{M}^*)$. Then there is some function $f: \prod_{n < n^{\ast}} \kappa^{<\omega}_{\alpha_n} \rightarrow \mathcal{P}(\mathbb{M}^*)$ such that for all $\vec{\eta} \in dom f$, $f(\vec{\eta})$ is a dense open subset of $\mathbb{M}^*$ and there are sequences $\vec{\rho}_n \in C_{{\alpha_n}}^{<\omega}$ for $n<n^\star$ such that $D = \bar{j}(f)(\vec{\rho}_0,\dots, \vec{\rho}_{n_\star - 1})$. We do assume that for every $n<n^\star$ the sequences $\vec{\rho}_n$ are an initial segments of the corresponding $C_{\alpha_n}$.

Let $M' = \{\bar{j}(g)(\vec{\rho}_0,\dots, \vec{\rho}_{n_\star - 1}) \mid g \in M\}$, where we only take the $g$'s that have the right domain, namely that $\langle \vec{\rho}_n \mid n < n_{\star-1}\rangle \in \bar{j}(\dom g)$. Working in $M'$, let us apply Lemma \ref{StrongPrikry} for $D$ and a condition with stem $\vec{\rho}=\langle \vec{\rho}_n \mid n < n_{\star-1}\rangle\frown \langle\emptyset\rangle$\footnote{This means the stem $s$ such that $s(\alpha_n)=\vec{\rho}_n$ and $s(\beta)=\emptyset$, otherwise} and let $p^\star$ and $\vec{\gamma}$ be the obtained direct extension and support. It is sufficient to show that $p^\star$ belongs to the filter $H$. Indeed, in such case $H$ will meet $D^\star=\{q\leq p^\star:\,\len(q)\leq_p \vec{\gamma}\}$ which is a subset of $D$. Let $g\colon \prod \kappa^{<\omega}_{\alpha} \to M$ be a function representing the sequence of large sets in $p^\star$ where $g(s)$ is of the form $\langle B^s_\alpha \mid \alpha < \mu\rangle$ and $s$ goes over stems with length sequence $\len p^\star$. Say $\vec{B}=\langle B^s_\alpha \mid \alpha < \mu, \, \len s=\len p^\star\rangle$. Using Lemma \ref{lemma: finite diagonal intersection}, we have a sequence of large sets $\vec{A}=\langle A_\alpha \mid \alpha < \mu\rangle$ such that for every $s\in\prod A^{<\omega}_\alpha$, $A_\alpha\setminus\max (s(\alpha))\subseteq B^s_\alpha$, for every $\alpha<\mu$. Clearly, the condition $p^{**}$ with stem $\vec{\rho}\frown \langle \emptyset\rangle$ and large sets $\bar{j}(\vec{A})$ is stronger than $p^*$. Let us verify that $p^{**}$ enters the generic and thus $p^*$ also.  

Let $\alpha<\mu$.  If $\alpha$ is no one of the $\alpha_n$'s then $p^{**}(\alpha)=\langle \emptyset, \bar{j}(\vec{A})_\alpha\rangle$. Let us show that $C_\alpha\subseteq \bar{j}(\vec{A})_\alpha$ and from this we will conclude that it compatible with all the conditions of $H_\alpha$.  By definition $j(\vec{A})_\alpha\in j(\vec{U})_\alpha$. Since $\alpha$ is a fixed point of $j_{\omega\alpha,\omega\mu}$ then $j(\vec{A})_{j_{\omega\alpha,\omega\mu}(\alpha)}\in j(\vec{U})_{j_{\omega\alpha,\omega\mu}(\alpha)}$ and hence $j_{\omega\alpha}(\vec{A})_\alpha\in j_{\omega\alpha}(\vec{U})_\alpha$. Thus by the definition of the iteration and since $\crit j_{\omega\alpha,\omega\alpha+1}>\alpha$,  $\rho^0_\alpha\in j_{\omega\alpha+1}(\vec{A})_\alpha$. The critical point of $j_{\omega\alpha+1,\omega\mu}$ is $\rho^1_\alpha$ an hence $\rho^0_\alpha$ and $\alpha$ are fixed by this embedding. This shows that $\rho^0_\alpha\in j(\vec{A})_\alpha$. Using the same argument one can show that $\rho^n_\alpha\in j(\vec{A})_\alpha$, for all $n\in\omega$.

Now let us suppose that $\alpha=\alpha_n$ for some $n<n^\star$ . We claim that $C_{\alpha_n}\setminus\vec{\rho}_n\subseteq \bar{j}(\vec{A})_{\alpha_n}$. Indeed, notice that $\bar{j}(\vec{A})_{\alpha_n}\setminus \max \vec{\rho}_n\subseteq j(\vec{B})^{\vec{\rho}}_{\alpha_n}\in j(\vec{U})_{\alpha_n}$. Thus $\bar{j}(\vec{A})_{\alpha_n}\setminus \max \vec{\rho}_n\in j(\vec{U})_{\alpha_n}$. On the other hand $\max \vec{\rho}_n=\max (\bar{j}''_{\omega\alpha_n+k+1, \omega\mu} \vec{\rho}_n)$ where $k=|\vec{\rho}_n|$. Hence $\bar{j}_{\omega\alpha_n+k+1}(\vec{A})_{\alpha_n}\setminus\max \vec{\rho}_n\in j_{\omega\alpha_n+k+1}(\vec{U})_{\alpha_n}$. By definition of the iteration, $\rho^{k+1}_{\alpha_n}\in j_{\omega\alpha_n+k+2}(\vec{A})_{\alpha_n}$ and hence $\rho^{k+1}_{\alpha_n}\in \bar{j}(\vec{A})_{\alpha_n}$ since $\crit j_{\omega\alpha_n+k+2,\omega\mu}>\rho^{k+1}_{\alpha_n}>\alpha_n$. Repeating this argument, we concluce that $C_{\alpha_n}\setminus\vec{\rho}_n\subseteq \bar{j}(\vec{A})_{\alpha_n}$. From this it is obvious that $p^{**}(\alpha_n)=\langle\vec{\rho}_n, \bar{j}(\vec{A})_{\alpha_n}\rangle\parallel q$, for each $q\in H_{\alpha_n}$. This completes the proof of genericity of $H$.
\end{proof}
Set $j^\star=\bar{j}\circ j$. The proof of next claim leads us to the end of the lemma.
\begin{claim}
The embedding $j^\star: V\rightarrow M^\star$ lifts to an elementary embedding $j^\star: V[G]\rightarrow M^\star[G\times H]$ which is a witness for $\lambda$-\Cn-supercompactness of $\kappa$ in $V[G]$.
\end{claim}
\begin{proof}[Proof of claim.]
Provide that $j^\star$ lifts, it is clear that this embedding will lie in $V[G]$ since $H$ is definable within $M$.
Let us first show that $j^\star$ lifts.  Let $p\in G$ and notice that $j(p)=p\frown q$ \footnote{This stands for the concatenation (in the right interpretation) of both conditions.} where $q\in \mathbb{M}^*$ has trivial stem. To be more precise, $q=\langle\langle(s(\alpha),B_\alpha\rangle:\,\alpha\in \mu\rangle$ such that $s(\alpha)=\emptyset$ and $B_\alpha\in {U}_{\alpha}$. Applying the second elementary embedding, we have that $p$ is not moved (since $\crit(\bar{j})>\kappa$) whereas $\bar{j}(q)=\langle\langle\emptyset,\bar{j}(\vec{B})_\alpha\rangle:\,\alpha\in \mu\rangle$\footnote{Here $\vec{B}$ stands for the sequence of large sets of $q$, $\langle A_\alpha:\,\alpha<\mu\rangle$.}.
For each $\alpha< \mu$, one can argue as in the proof of genericity for $H$ that $C_\alpha\subseteq \bar{j}(\vec{B})_\alpha$ and thus $\langle \emptyset, \bar{j}(\vec{B})_\alpha\rangle\parallel q$, for all $q\in H_\alpha$. In particular, $\bar{j}(q)\in H$ and hence $j^\star(p)\in G\times H$.

To finish the claim it remains to show that $N=M^\star[G\times H]$ is closed by $\lambda$-sequences since $j^\star(\kappa)=j(\kappa)\in C^{(n)}$ because $\mathbb{M}$ is mild. Since $N$ is a model of choice, it is sufficient to show that every $\lambda$-sequence of ordinals from $V[G]$ belong to $N$. Note that the forcing $\mathbb{M}$ introduces new $\omega$-sequences. First, since $j$ is a $\lambda$-supercompact embedding, $M$ is closed under $\lambda$-sequences from $V$. Let $\sigma \in V$ be an $\mathbb{M}$-name for a $\lambda$-sequence of ordinals. By the $\kappa$-chain condition of $\mathbb{M}$, we may assume that $|\sigma| = \lambda$ and that $\sigma \subseteq \mathbb{M} \times \mathrm{ON}$. Therefore $\sigma \in M$, and in $M[G]$ we can interpret it. Let us finally show that $M[G]$ and $M^\star[G\times H]$ contain the same $\lambda$-sequences of ordinals. 

Let $\langle \xi_\alpha :\, \alpha < \lambda \rangle$ be a sequence of ordinals. In $M^\star$, for every $\alpha$ there is a function $f_\alpha$ such that $\bar{j}(f_\alpha)(\vec{\rho}^\alpha_0,\dots, \vec{\rho}^\alpha_{n^\alpha-1}) = \xi_\alpha$, where $\vec{\rho}^\alpha_i$ is a finite sequence of elements of $C_{\zeta_i}$, some $\zeta_i < \mu$.  Since the critical point of $\bar{j}$ is above $\lambda$ and the sequence of functions $\langle f_\alpha \mid \alpha < \lambda\rangle \in M[G]$ we conclude that $\bar{j}(\langle f_\alpha \mid \alpha < \lambda\rangle) = \langle \bar{j}(f_\alpha) \mid \alpha < \lambda\rangle \in M^\star[G]$. Thus, it is sufficient to show that the sequence $\vec{R} = \langle \langle \vec{\rho}^\alpha_i \mid i < n^\alpha\rangle \mid \alpha < \mu\rangle \in M^\star[G][H]$. 

Let us define by induction on $\gamma < \mu$ a sequence of functions $p_\gamma$ such that $\bar{j}(p_\gamma)(H) = \gamma$. Intuitively, $p_\gamma$ is a procedure for extracting $\gamma$, given the information of $H$. Let us assume that $p_\beta$ is defined for all $\beta < \gamma$. 
Since the critical point of $j_{\gamma, \mu}$ is above $\gamma$, we know that $\gamma$ is represented in $M_\gamma$ by 
\[\gamma = j_{0,\gamma}(g)(\rho_0, \dots, \rho_{n-1}),\]
for some elements of the sequences in $H$, $\rho_0, \dots, \rho_{n-1}$. Those elements are all below the $\gamma$-th member of $H$ in the increasing enumeration and in particular, do not move under $j_{\gamma, \mu}$. Let $h\colon \mu \to \mu$ be the increasing enumeration of $H$. 
Let $\beta_0, \dots, \beta_{n-1}$ be their indices, so $h(\beta_i) = \rho_i$. We conclude that: 
\[\gamma = \bar{j}(g)(h(p_{\beta_0}(H)), \dots, h(p_{\beta_{n-1}}(H)) ),\]   
so we can define $p_\gamma$. 

Finally, let us show that the sequence $\vec{R}$ is in $M^\star[G][H]$. Indeed, one can obtain $\vec{R}$ from $H$ by just knowing the indices of each $\vec{\rho}^\alpha_i$. This sequence of indices is equivalent to a sequence of ordinals below $\mu$ of length $\lambda$, $\vec{\epsilon} = \langle \epsilon_\alpha \mid \alpha < \lambda\rangle$. Letting the condition $\vec{p} = \langle p_{\epsilon_\alpha} \mid \alpha < \lambda\rangle \in M$. and applying the components of $\bar{j}(\vec{p})$ to $H$ we obtain $\vec{\epsilon}$. Finally, applying $h$ on the components of $\vec{\epsilon}$, we obtain $\vec{R}$, as wanted. 
\end{proof}
\end{proof}
This immediately yields to the proof theorem \ref{MainTheorem}.
\begin{proof}[Proof of theorem \ref{MainTheorem}]
By results of D\v{z}amonja and Shelah \cite{DzjaSh}, it is known that if one changes the cofinality of some inaccessible cardinal $\delta$ to $\omega$ but preserves its successor then $\square_{\delta,\omega}$ holds in the generic extension. Consequently, $\mathbb{M}$ adds unboundely many $\square_{\delta,\omega}$-sequences below $\kappa$ and thus there is no ($\omega_1$-)strongly compact cardinal below it. Combining this with lemma \ref{MainLemma} we are done.
\end{proof}
To conclude this section we would like to point out that the ideas used in the proof of claim \ref{ClaimGenericity} can be straightforwardly adapted to proof the following version of Mathias criteria for the Magidor product of Prikry forcings:
\begin{theo}[Mathias criteria]
Suppose that $M$ is an inner model of ZFC and \linebreak$\langle U_\alpha:\alpha<\kappa\rangle$ is a sequence of normal measures over the cardinals $\langle \kappa_\alpha:\alpha<\kappa\rangle$, respectively. A sequence $\vec{C}\in\prod_{\alpha\in\kappa} {^{\kappa_\alpha}\kappa_\alpha}$ defines a generic filter for $\mathbb{M}$ if it satisfies the following condition:
$$\forall\alpha\in\kappa\,\forall A\in U_\alpha\,|\vec{C}(\alpha)\setminus A|<\aleph_0.$$
Moreover, the generic is given by
$$G(\vec{C})=\{p\in\mathbb{M}:\,\forall \alpha\in \kappa\, (p(\alpha)=\langle s(\alpha), A_\alpha\rangle\wedge s(\alpha)\vartriangleleft \vec{C}(\alpha)\,\wedge\, \vec{C}(\alpha)\setminus\max s(\alpha)+1\subseteq  A_\alpha)\} $$
\end{theo}

\subsection{Some consequences of theorem \ref{MainTheorem}}
In this section we shall analyse some of the consequences of theorem \ref{MainTheorem}. For each $n\geq 1$ let us respectively denote by $\Gamma_n$ and by $\Gamma^*_n$ the first order formulas
\begin{eqnarray*}
\text{``$\min\mathfrak{M}<\min\mathfrak{K}_{\omega_1}=\min \mathfrak{K}=\min \mathfrak{S}=\min \mathfrak{S}^{(n)}$ ''}\\
\text{``$\min\mathfrak{M}<\min\mathfrak{K}_{\omega_1}=\min \mathfrak{K}=\min \mathfrak{S}=\min \mathfrak{S}^{(n)}<\min\mathfrak{E}$ ''.}
\end{eqnarray*}
\begin{cor}
For every $n\geq 1$, 
$$\Con(\mathrm{ZFC}+\exists \kappa\,(\text{$\kappa$ \Cn-extendible}))\rightarrow \Con(\mathrm{ZFC}+\Gamma_n).$$
In particular, for every $n\geq 3$
$$\Con(\mathrm{ZFC}+\exists \kappa\,(\text{$\kappa$ \Cn-extendible}))\rightarrow \Con(\mathrm{ZFC}+\Gamma^\star_n).$$
\end{cor}
\begin{proof}
The first claim follows automatically from theorem \ref{MainTheorem}. For the second claim it will suffice to show that the existence of a \Cn-extendible cardinal entails the existence of an extendible cardinal above. Indeed, let $\kappa$ be a \Cn-extendible and notice that for every $\alpha<\kappa$ the formula $\varphi(\alpha)$
$$\text{``$\exists\beta\,(\beta>\alpha\,\wedge\,\beta\,\text{extendible}\,)$''}$$
is true and $\Sigma_4$, hence, $V_\kappa\models \text{``$\forall\alpha\,\varphi(\alpha)$''}$. Since \Cn-extendible cardinals are $C^{(5)}$-correct (see e.g. \cite{Bag}), the formula $\text{``$\forall\alpha\,\varphi(\alpha)$''}$ is already true and thus there is a proper class of extendible cardinals in the universe.
\end{proof}
\begin{remark}
\rm{New results due to the third author and Woodin have pointed out that any \Cn-extendible cardinal is a limit of \Cn-supercompact. In particular, the second claim of the corollary is already true for any $n\geq 1$. }
\end{remark}

At the light of theorem \ref{MainTheorem} the identity crises for \Cn-supercompact cardinals turns to be a plausible scenario. One may even ask if this result may be strengthened or, more particularly, if the ultimate identity crises for \Cn-supercompact cardinals is consistent; namely, provided it exists, if the first \Cn-supercompact cardinal, for each $n\geq 1$, can be the first ($\omega_1$)-strongly compact cardinal.
On this respect, the natural large cardinal hypothesis to start with is the existence of a $C^{(<\omega)}$-extendible cardinal: namely, a cardinal $\kappa$ which is \Cn-extendible, for each $n\geq 1$. Notice however that, by Tarski's theorem of undefinability of truth, the existence of such cardinals can not be expressed by a first order formula but via a countable schema of first order formulae. Let $\mathbf{k}$ be a constant symbol and consider the language of set theory augmented with it, $\mathcal{L}=\{\in,\mathbf{k}\}$.   
\begin{defi}
We will denote by $C^{(<\omega)}-\mathrm{EXT}$ the countable schema of first order formulae asserting that for each (meta-theoretic) $n\in\omega$ the $\mathcal{L}$-formula
$\text{`` $\mathbf{k}$ is \Cn-extendible''}$
holds. If $\mathfrak{M}=\langle M,\in,x\rangle$ is a $\mathcal{L}$-structure, we agree that the interpretation of the constant symbol $\mathbf{k}$ is $x$. We will write $\mathfrak{M}\models C^{(<\omega)}-\mathrm{EXT}$ if for every (meta-theoretic) $n\in\omega$ the formula $\text{``$\mathfrak{M}\models \text{$\mathbf{k}$ is \Cn-extendible}$''}$ is true. We will also denote by $\mathrm{ZFC}^\star$ the version of all $\ZFC$ axioms where we allow a constant symbol $\mathbf{k}$ to be used in any instance of axioms of replacement and separation.
\end{defi}

\begin{defi}[\Co-extendible cardinal]
Let $\kappa$ be a cardinal and $\mathfrak{M}=\langle M,\in, \kappa\rangle$ be a $\mathcal{L}$-structure. We will say that $\kappa$ is $\mathfrak{M}$-\Co-extendible if $\mathfrak{M}\models C^{(<\omega)}-\mathrm{EXT}$. If $\mathfrak{M}=\langle V, \in, \kappa\rangle$ we will simply say that $\kappa$ is \Co-extendible.
\end{defi}
In a analogous way, we can define the schema $C^{(<\omega)}-\mathrm{SUP}$ for the intended notion of $C^{(<\omega)}$-supercompactness. Let $\mathfrak{C}^{(<\omega)}$ and $\mathfrak{S}^{(<\omega)}$ denote the class of \Co-extendible and \Co-supercompact cardinals, respectively.

By results of Bagaria \cite{Bag}, the schema $C^{(<\omega)}-\mathrm{EXT}$ implies that \emph{Vop\v{e}nka Principle} holds. Recall that given $\kappa<\lambda$ the cardinal $\kappa$ is called $\lambda$-\emph{superhuge} if there is an elementary embedding $j:V\rightarrow M$ such that $\crit(j)=\kappa$, $j(\kappa)>\lambda$ and $M^{j(\kappa)}\subseteq M$. If $\kappa$ is $\lambda$-\emph{superhuge} for each $\lambda>\kappa$, the cardinal $\kappa$ is called \emph{superhuge}. If we are given a cardinal $\theta$, we will say that \emph{$\theta$ is a target of $\kappa$} (\emph{$\kappa\rightarrow (\theta)$}) when there is some ordinal $\lambda>\kappa$ and some $\lambda$-superhuge embedding $j: V\rightarrow M$ such that $j(\kappa)=\theta$. It is known that if $\kappa$ is superhuge then the collection of all of its targets is a proper class.

In \cite{BarDipri} the authors introduced an strengthening of the classical notion of superhugness. A cardinal $\kappa$ is \emph{stationarily superhuge} if its collection of targets forms a stationary proper class\footnote{Again, this notion is not first order expressible.}. Since for every $n\in\omega$ the class \Cn is a club class it is obvious that any model with an stationarily superhuge cardinal $\kappa$ satisfies the schema $C^{(<\omega)}-\mathrm{EXT}$ as witnessed by $\kappa$. As pointed out in theorem 6b of the aforementioned paper, the consistency strenght of a stationarily superhuge cardinal is below the consistency of a $2$-huge cardinal. Therefore the consistency strength of the schema $C^{(<\omega)}-\mathrm{EXT}$ is bounded by below by \VP and by above by the existence of a $2$-huge cardinal.

Let $\kappa$ be a \Co-extendible cardinal.
By Tsaprounis' result \cite{Tsan}, for each $n\geq 1$ there is a  $\mathfrak{E}^{(n)}$-fast function $\ell_n:\kappa\rightarrow\kappa$ in $V$. Notice that $V_\kappa \prec V$ and thus one can define those functions uniformly in $V_{\kappa + 1}$, so the function $\ell = \sup \ell_n$ is a member of $V$. Arguing as in theorem \ref{MainTheorem} the ultimate identity crises theorem follows:
\begin{theo}
Let $\langle V,\in,\kappa\rangle$ be a model of (a large enough fragment of) $\mathrm{ZFC}^\star$ plus $C^{(<\omega)}-\mathrm{EXT}$. Then in the generic extension $V^\mathbb{M}$ the where the chain of relations
$$\min\mathfrak{M}<\min\mathfrak{K}_{\omega_1}=\min\mathfrak{K}=\min\mathfrak{S}=\min \mathfrak{S}^{(<\omega)}<\min\mathfrak{E} $$
holds.
\end{theo}
This immediately yields to the following corollary:
\begin{cor}
\[\Con(\mathrm{ZFC}+\exists\,\kappa\,(\kappa\text{ is }2\text{-huge}))\rightarrow \Con(\mathrm{ZFC}+\Xi)\]
where $\Xi$ is the scheme 
\begin{center}
``$\min\mathfrak{M}<\min\mathfrak{K}_{\omega_1}=\min\mathfrak{K}=\min\mathfrak{S}=\min\mathfrak{S}^{(<\omega)}<\min\mathfrak{E}$''.
\end{center}
\end{cor}

\section{A summary of what is known}
In the present section we shall briefly summarize all the known consistency relations between the classes of supercompact, \Cn-supercompact and \Cn-extendible cardinals. Similarly to the classical Magidor's-like analysis of supercompact cardinals in this setting there are also two critical scenarios: the first one corresponding to the identity crises phenomenon  discussed in previous sections and the second one where the expected hierarchic relations between large cardinals hold.

As pointed out earlier, the case of \Cn-extendible cardinals is paradigmatic in the sense that they are not affected by the identity crises pathology. In other words, the class of \Cn-extendibles is ordered hierarchically  and thus its configuration fits within the second paradigm of the universe described so far. Nonetheless the situation with respect to \Cn-supercompact cardinals may be completely different by virtue of theorems \ref{MainTheorem} and  \ref{MainTheorem2}. Specifically, we have shown that an extreme identity crises for these classes of cardinals is possible by making the first \Co-supercompact cardinal the first ($\omega_1$-)strongly compact cardinal.

Recent investigations of the third author with Woodin have brought to light that the antagonistic scenario is also possible under the assumption of a new axiom called EEA \cite{PW}.

\begin{axiom}[Extender Embedding Axiom (EEA)]
Let $j:V \to M$ be an elementary embedding with critical point $\kappa$ such that $j(\kappa)$ is a limit cardinal and such that $M$ is closed under $\omega$-sequences. Then $\kappa$ is $j(\kappa)$-superstrong, i.e $V_{j(\kappa)}\subseteq M$.
\end{axiom}
The point for EEA is that under this axiom the configuration of the different classes $\mathfrak{S}^{(n)}$ coincide with the standard ordering pattern of the large cardinal hierarchy: 
\begin{theo}[P.-Woodin \cite{PW}]\label{PovedaWoodin}
Assume EEA. Then the following clauses hold:
\begin{enumerate}
\item For each $n\geq 1$, the class of \Cn-supercompact cardinal is included in $C^{(n+2)}$.\footnote{This is optimal as being \Cn-supercompact is a $\Pi_{n+2}$ property.}
\item For each $n\geq 1$, $\text{``$\min \mathfrak{S}<\min\mathfrak{S}^{(n)}<\min\mathfrak{E}^{(n)}<\min\mathfrak{S}^{(n+1)}$''}$ holds.
\end{enumerate}
\end{theo}
It is worth to emphasize that the inequality $\text{``$\min\mathfrak{S}^{(n)}<\min\mathfrak{E}^{(n)}$''}$ is proved without need of EEA, though. Altogether, it seems that EEA is the right axiom one has to consider to force the universe to have the \textit{expected} configuration in the section of the large cardinal hierarchy ranging between the first supercompact cardinal and \VP. Therefore it turns out that a central issue for the study of such cardinals is to clarify the status of EEA modulo large cardinals: namely if it is already consistent. On this respect the present paper has implicitly made some steps towards solving this issue. More precisely, at the light of theorem \ref{PovedaWoodin}, EEA can not coexists with the identity crises phenomenon and thus it must fails in the Magidor's model discussed in the previous section.
Nowadays the study of the consistency of EEA forms part of an ongoing project between the third autor and Woodin and it seems it has deep connections with the inner model program at finite  levels of supercompactness.

\section{Open Questions and concluding remarks}\label{OpenQuestions}
We would like to conclude the present paper exposing certain questions of combinatorial flavour that remain open. Broadly speaking we are interested to answer, with the most possible generality, the following question:

\begin{quest}
\rm{What can be said about the combinatorics of $V$ under the existence of \Cn-supercompact cardinals? }
\end{quest}
 
Unlike supercompact cardinals it does not seem evident how to develop a theory that stu\-dies the consequences of \Cn-supercompact cardinals on the combinatorics of $V$.
In the context of supercompact cardinal this project has been carried out successfully, mainly by means of the method of forcing,  yielding to a rich and vast theory. There are many paradigmatic examples on this respect but one of the most important is the Laver's theorem of indestructibility of supercompact cardinals by $\kappa$-directed closed forcing \cite{Lav}. Speaking in general, Laver's result shows that supercompactness is a robust notion with respect to a wide family of (set) forcings where one can find $\Add(\kappa,\lambda)$ among many others. In particular, Laver's theorem shows that supercompactness is consistent with any prescribed behaviour of the power set function on $\kappa$.\footnote{There are also similar results with partial square principles as pointed out in previous sections.} The moral here is that one can get relevant information about the combinatorics of $V$ from the robustness of supercompactness.

Nevertheless,  this does not seem to be the case for the class of \Cn-supercompact cardinals. For instance, as commented in former sections, it is not evident whether these cardinals carry $\mathfrak{S}^{(n)}$-fast functions and thus one can not naively adapt Laver's indestructibility arguments to this new setting. In fact theorem \ref{PovedaWoodin} indicates that under EEA any \Cn-supercompact cardinal is a $C^{(n+2)}$-cardinal hence no indestructibility result is available for such cardinals \cite{BagHam}. This suggest the following question:



\begin{quest}
\rm{
Let $\kappa$ be a  \Cn-supercompact cardinal.
What kind of forcings preserve the \Cn-supercompactnes of $\kappa$? For instance, is it possible to add many Cohen susbsets to $\kappa$ while preserving its \Cn-supercompactness?}
\end{quest}
In the next lines we will give an outline of the main difficulties one faces up with discussing the interplay of forcing with \Cn-supercompact cardinals.
Speaking in general, for any given forcing there are two standard ways to proceed on this respect: either analysing under which hypothesis the corresponding embeddings may be lifted or how can one define extenders witnessing the \Cn-supercompactness of $\kappa$ in the generic extension. In the next lines we shall try to argue that any of both strategies seem non trivial to implement.

Let $\mathbb{P}$ be a forcing notion, $G\subseteq \mathbb{P}$ a generic filter, $\lambda>\kappa$ be an arbitrary cardinal and $j:V\rightarrow M$ be an elementary embedding witnessing the $\lambda$-$C^{(n)}$-supercompactness of $\kappa$. The strategies previously commented may be phrased in the following terms:
\begin{itemize}
\item[$\diamondsuit$] \textbf{Lifting strategy:} Lift $j$ to $j^*$ witnessing the $\lambda$-$C^{(n)}$-supercompactness of $\kappa$ in $V[G]$.
\item[$\diamondsuit$]\textbf{Extender strategy:} Use $j$ to define in $V[G]$ an extender $E$ such that $j_E: V[G]\rightarrow M$ witnesses the $\lambda$-$C^{(n)}$-supercompactness of $\kappa$ (see section 5 of \cite{Bag} for details) .
\end{itemize}
Notice that regardless of the strategy the cardinal $j(\kappa)$ remains in the class $(C^{(n)})^{V[G]}$ since the forcing $\mathbb{P}$ is mild.
\subsubsection{\textbf{Lifting strategy}}
If $\mathbb{P}$ is a $\kappa$-Easton support iteration of forcings within $V_\kappa$ it is not hard to show that $j$ lifts to $j^*:V[G]\rightarrow M[G\ast H]$, where $H\subseteq j(\mathbb{P})/\mathbb{P}$ is generic over $M[G]$. Furthermore, with a bit of care, one may make sure that $M[G\ast H]^\lambda\subseteq M[G\ast H]$.\footnote{For instance guiding $\mathbb{P}$ with some fast function as we did in the proof of Proposition \ref{PreservationOfSuper}}

 Thereby the main issue here is how to ensure that $j^*$ is definable in $V[G]$ or, in other words, that the $M[G]$-generic filter $H$ lies in $V[G]$. There are specific situations where one can argue on this direction; for instance, using a diagonalization argument as in Proposition 8.1 of \cite{CumIter} or appealing to the distributiviness of the tail forcing $j(\mathbb{P})/\mathbb{P}$ as in Lemma 3.5 in \cite{TsaPhD}. Nonetheless both arguments rely in the fact that whilst $j(\kappa)$ is very large in $M$ it is small in $V$. It is clear that this is never the case for \Cn-supercompact cardinals.

Consequently the \textbf{Lifting strategy} yields to the issue of \textbf{building definable generics for} $j(\mathbb{P})/\mathbb{P}$ which suggests that one has to be able to \textit{handmade} generics for $j(\mathbb{P})/\mathbb{P}$ . Notice that this is precisely the procedure we have followed in the proof of theorem \ref{MainTheorem}.
\subsubsection{\textbf{Extender strategy}}
This strategy is used for instance in Proposition \ref{PreservationOfSuper} or Lemma 6.4 of \cite{GitPrikry}.  Assume $\mathbb{P}$ is a forcing $\kappa$-iteration of forcings within $V_\kappa$  with a close enough tail forcing $j(\mathbb{P})/\mathbb{P}$. Lift $j$ to $j^\star: V[G]\rightarrow M[G\ast H]$ as before and afterwards define $E$ to be the \textit{potential extender} derived from $j^*$. More precisely, set $E=\langle E_a:\,a\in [\eta]^{<\omega} \rangle$ as
$$(\star)\;\;\;X\in E_a\;\longleftrightarrow\;\exists p\in G\,\exists q\leq j (p)\setminus\kappa, \; p\frown q\Vdash_{j(\mathbb{P})} \dot{a}\in j(\dot{X})$$
where $\dot{a},\dot{X}$ are $\mathbb{P}$-names and $\eta$ is some ordinal. Here the closedness of the tail is used to argue that $E\in V[G]$. 

As it is shown in \cite{GitPrikry} if $\mathbb{P}$ is a suitable Prikry-type iteration and the order relation appearing in $(\star)$ is $\leq^\star$ then $E_a$ is a $\kappa$-complete normal measure, each $a\in[\eta]^{<\omega}$. The main issue here thus is not related with the definability of the extender nor with its combinatorial properties but with $j_E(\kappa)$. Notice that we have to make sure that $j_E(\kappa)$ is a $C^{(n)}$-cardinal in $V[G]$ and thus it is natural to ask whether $j_E(\kappa)=j(\kappa)$. Nonetheless this technical point seems very hard to fulfil due to the \textit{generic definition} of $E$. In summary, the \textbf{Extender strategy} yields to the the issue of \textbf{finding extenders $E$ such that $j_E(\kappa)=j(\kappa)$.}

\section*{Acknowledgments}
\rm{The present paper has been prepared during a research stay of the third author in the Einstein Institute of Mathematics at the Hebrew University of Jerusalem. The third author wants to express his gratitude to professor S. Shelah for his kindness for inviting him and to professors J. Bagaria, W. H. Woodin and M. Gitik for many illuminating discussions on the matter. In the same way, the third author extends his gratitude to the Einstein Institute of Mathematics for his warmly hospitality during his stay.}
\bibliographystyle{alpha} 
\bibliography{biblio}

\begin{thebibliography}{BCMR15}

\bibitem[Apt05]{Apter}
Arthur~W. Apter.
\newblock Diamond, square, and level by level equivalence.
\newblock {\em Archive for Mathematical Logic}, 44(3):387--395, 2005.

\bibitem[Apt06]{ApterIndestr}
Arthur~W. Apter.
\newblock The least strongly compact can be the least strong and
  indestructible.
\newblock {\em Annals of Pure and Applied Logic}, 144(1):33 -- 42, 2006.

\bibitem[Bag12]{Bag}
Joan Bagaria.
\newblock ${C^{(n)}}$-cardinals.
\newblock {\em Archive for Mathematical Logic}, 51(3):213--240, 2012.

\bibitem[BBT13]{BagBroo}
Joan Bagaria and Andrew Brooke-Taylor.
\newblock On colimits and elementary embeddings.
\newblock {\em The Journal of Symbolic Logic}, 78(2):562--578, 2013.

\bibitem[BCMR15]{BagEtAl}
Joan Bagaria, Carles Casacuberta, Adrian~RD Mathias, and Ji{\v{r}}{\'\i}
  Rosick{\`y}.
\newblock Definable orthogonality classes in accessible categories are small.
\newblock {\em Journal of the European Mathematical Society}, 17(3):549--589,
  2015.

\bibitem[BDT84]{BarDipri}
Julius~B Barbanel, Carlos~A Diprisco, and It~Beng Tan.
\newblock Many-times huge and superhuge cardinals.
\newblock {\em The Journal of symbolic logic}, 49(1):112--122, 1984.

\bibitem[BHTU16]{BagHam}
Joan Bagaria, Joel~David Hamkins, Konstantinos Tsaprounis, and Toshimichi
  Usuba.
\newblock Superstrong and other large cardinals are never laver indestructible.
\newblock {\em Archive for Mathematical Logic}, 55(1-2):19--35, 2016.

\bibitem[BM14a]{BM}
Joan Bagaria and Menachem Magidor.
\newblock Group radicals and strongly compact cardinals.
\newblock {\em Transactions of the American Mathematical Society},
  366(4):1857--1877, 2014.

\bibitem[BM14b]{BM2}
Joan Bagaria and Menachem Magidor.
\newblock On $\omega_1$-strongly compact cardinals.
\newblock {\em J. Symb. Log.}, 79(1):266--278, 2014.

\bibitem[BP18]{BP}
Joan Bagaria and Alejandro Poveda.
\newblock More on the preservation of large cardinals under class forcing.
\newblock {\em arXiv preprint arXiv:1810.09195}, 2018.

\bibitem[CFM01]{CummingsSquare}
James Cummings, Matthew Foreman, and Menachem Magidor.
\newblock Squares, scales and stationary reflection.
\newblock {\em Journal of Mathematical Logic}, 1(01):35--98, 2001.

\bibitem[Cum10]{CumIter}
James Cummings.
\newblock Iterated forcing and elementary embeddings.
\newblock In {\em Handbook of set theory}, pages 775--883. Springer, 2010.

\bibitem[DS]{DzjaSh}
Mirna D{\v{z}}amonja and Saharon Shelah.
\newblock On squares, outside guessing of clubs and ${I}_{<f}[\lambda]$.
\newblock {\em Fundamenta Mathematicae}, 148(2).

\bibitem[Git10]{GitPrikry}
Moti Gitik.
\newblock Prikry-type forcings.
\newblock In {\em Handbook of set theory}, pages 1351--1447. Springer, 2010.

\bibitem[Ham00]{HamLott}
Joel~David Hamkins.
\newblock The lottery preparation.
\newblock {\em Ann.~Pure Appl.~Logic}, 101(2-3):103--146, 2000.

\bibitem[Kun14]{Kun}
Kenneth Kunen.
\newblock {\em Set theory an introduction to independence proofs}, volume 102.
\newblock Elsevier, 2014.

\bibitem[Lav78]{Lav}
Richard Laver.
\newblock Making the supercompactness of $\kappa$ indestructible under
  $\kappa$-directed closed forcing.
\newblock {\em Israel Journal of Mathematics}, 29(4):385--388, 1978.

\bibitem[Mag71]{Mag}
M.~Magidor.
\newblock On the role of supercompact and extendible cardinals in logic.
\newblock {\em Israel Journal of Mathematics}, 10(2):147--157, Jun 1971.

\bibitem[Mag76]{MagSuper}
Menachem Magidor.
\newblock How large is the first strongly compact cardinal? {Or} a study on
  identity crises.
\newblock {\em Annals of Mathematical Logic}, 10(1):33--57, 1976.

\bibitem[PW]{PW}
Alejandro Poveda and Hugh~W. Woodin.
\newblock {${C^{(n)}}$}-cardinals, inner models and the eea axiom.
\newblock {\em Unpublished draft}.

\bibitem[Sol74]{Sol}
Robert~M. Solovay.
\newblock Strongly compact cardinals and the {GCH}.
\newblock pages 365--372, 1974.

\bibitem[SRK78]{SolReinKan}
Robert~M Solovay, William~N Reinhardt, and Akihiro Kanamori.
\newblock Strong axioms of infinity and elementary embeddings.
\newblock {\em Annals of Mathematical Logic}, 13(1):73--116, 1978.

\bibitem[Tsa]{Tsan}
Konstantinos Tsaprounis.
\newblock On {${C^{(n)}}$}-extendible cardinals.
\newblock {\em To appear in Journal of Symbolic Logic}.

\bibitem[Tsa12]{TsaPhD}
Konstantinos Tsaprounis.
\newblock Large cardinals and resurrection axioms.
\newblock 2012.

\bibitem[Tsa13]{Tsa}
Konstantinos Tsaprounis.
\newblock On extendible cardinals and the {GCH}.
\newblock {\em Archive for Mathematical Logic}, 52(5-6):593--602, 2013.

\bibitem[Tsa14]{TsaChain}
Konstantinos Tsaprounis.
\newblock Elementary chains and ${C^{(n)}}$-cardinals.
\newblock {\em Archive for Mathematical Logic}, 53(1-2):89--118, 2014.

\bibitem[Tsa15]{TsaResu}
Konstantinos Tsaprounis.
\newblock On resurrection axioms.
\newblock {\em The Journal of Symbolic Logic}, 80(2):587--608, 2015.

\end{thebibliography}
\end{document}